\documentclass[12pt,a4paper]{amsart}
\usepackage{amssymb}
\usepackage{latexsym}
\usepackage{exscale}
\usepackage{amsfonts}
\usepackage{graphicx}
\usepackage{subfigure}
\usepackage{mathrsfs}
\usepackage{amsmath,amscd,amsthm}
\usepackage{bbm,pifont}
\usepackage{enumerate}
\usepackage{color}
\usepackage{amssymb}
\usepackage{latexsym}
\usepackage{exscale}
\usepackage{setspace}
\usepackage{hyperref}

\allowdisplaybreaks

\DeclareMathAlphabet\mathcalbf{OMS}{cmsy}{b}{n}
\DeclareMathAlphabet\EuScript{U}{eus}{m}{n}
\DeclareMathAlphabet\EuScriptBold{U}{eus}{b}{n}

\numberwithin{equation}{section}
\newtheorem{theorem}{Theorem}[section]

\newtheorem{corollary}[theorem]{Corollary}
\newtheorem{proposition}[theorem]{Proposition}

\newtheorem{definition}[theorem]{Definition}

\newtheorem{remark}[theorem]{Remark}

\def\a{\alpha}

\def\rn{\mathbb{R}^n}
\def\lam{\lambda}
\def\d{\delta}
\def\les{\lesssim}
\def\v{\varepsilon}
\def\bp{\begin{proof}}
\def\ep{\end{proof}}

\begin{document}
\allowdisplaybreaks

\title[]
{Fractional Fourier series on the torus and applications}
\author{Chen Wang, Xianming Hou, Qingyan Wu, Pei Dang and Zunwei Fu}

\address{Chen Wang,
School of Mathematics and Statistics, Linyi University, Linyi 276000, China}
\email{18678705228@163.com}

\address{Xianming Hou,
School of Mathematics and Statistics, Linyi University, Linyi 276000, China}
\email{houxianming37@163.com}

\address{Qingyan Wu,
School of Mathematics and Statistics, Linyi University, Linyi 276000, China}
\email{wuqingyan@lyu.edu.cn}

\address{Pei Dang,
Faculty of Innovation Engineering, Macau University of Science and Technology, Macau, China}
\email{pdang@must.edu.mo}

\address{Zunwei Fu,
College of Information Technology, The University of Suwon, Bongdameup, Hwaseong-si,
Gyeonggi-do, 445-743, Korea}
\email{zwfu@suwon.ac.kr}


\begin{abstract}
In this paper, we introduce the fractional Fourier series on the fractional torus and proceed to investigate several fundamental aspects. Our study includes topics such as fractional convolution, fractional approximation, fractional Fourier inversion, and the Poisson summation formula. We also explore the relationship between the decay of fractional Fourier coefficients and the smoothness of a function. Additionally, we establish the pointwise convergence of the fractional Fourier series using the properties of the fractional F\'{e}jer kernel. Finally, we demonstrate the practical applications of the fractional Fourier series, particularly in the context of solving fractional partial differential equations with periodic boundary conditions, and showcase the utility of approximation methods on the fractional torus for recovering non-stationary signals.
\end{abstract}

\maketitle


\section{Introduction and statement of main results}
\setcounter{equation}{0}

It is well known that the Fourier series plays a crucial role in studying boundary value problems, like the vibrating string and the heat equation. Fourier's approach to solving the problem of heat distribution in the cube $\mathbf{T}^3$ by using the triple sine series of three-variable functions. Dirichlet \cite{Dirichlet} made significant contributions by studying the pointwise convergence of the Fourier series for piecewise monotonic functions in terms of the certain kernel on the circle. Inspired by the periodicity observed in astronomical and geophysical phenomena, many scholars have dedicated their efforts to studying expansions of periodic functions, as mentioned in \cite{van}. The relevance of these expansions in both theoretical frameworks and practical applications underscores the importance of analyzing Fourier series properties on the torus.

The $n$-torus $\mathbf{T}^n$ is defined as the cube $[0,1]^n$ with opposite sides identified. It is important to note that functions on $\mathbf{T}^n$ exhibit periodicity with a period of 1 in every coordinate. The $m$-th Fourier coefficient can be defined using the Fourier transform as follows:
\begin{align*}
\mathcal{F}(f)(m)=\int_{\mathbf{T}^n} f(x)e^{-2\pi i m\cdot x}dx,
\end{align*}
where $f\in L^1(\mathbf{T}^n)$, $m\in \mathbf{Z}^n$. The Fourier series of $f$ is the series
\begin{align}\label{eq1.1}
\sum_{m\in \mathbf{Z}^n}\mathcal{F}(f)(m)e^{2\pi i m\cdot x}.
\end{align}
Kolmogorov \cite{Kolmogorov1, Kolmogorov2} demonstrated the existence of a function $f\in L^1(\mathbf{T}^1)$ whose Fourier series diverges almost everywhere. For $1<p<\infty$, Carleson and Hunt \cite{Carleson, Hunt} observed that the Fourier series of an $L^p(\mathbf{T}^1)$ function converges almost everywhere. Fefferman \cite{Fefferman} presented an alternative proof for this result. Lacey and Thiele \cite{LT}, by combining ideas from \cite{Carleson} and \cite{Fefferman}, introduced a third approach to Carleson's theorem on the pointwise convergence of the Fourier series, providing valuable insights into the subject.
For a comprehensive understanding of the essential properties and applications of the Fourier series on the torus, one can refer to \cite{Victor, G, HL, MZ} and the references therein.

With the development of signal processing, the Fourier transform has revealed limitations in its ability to handle non-stationary signals. To address this challenge, the fractional Fourier transform (FRFT) was introduced. In this situation, the fractional Fourier transform was proposed to overcome this problem.
Additionally, when a signal $f$ exhibits $t$-periodicity with $t<1$, it is not well-defined on $\mathbf{T}^n$. In such cases, we can not obtain its frequency components. Therefore, we need to introduce the matched fractional torus to study such a kind of signal (or function).
\begin{definition}
Let $\a\in \mathbb{R}$, $\a\neq \pi \mathbf{Z}$. The fractional torus of order $\a$, denoted by $\mathbf{T}^n_\a$, is the cube $[0,|\sin\a|]^n$ with opposite sides identified.
\end{definition}
\noindent Note that $\mathbf{T}^n_\a=\mathbf{T}^n$ for $\a=\pi/2+2\pi\mathbf{Z}$, see Figure \ref{fig:1.1} (a) for two-dimensional torus. The graphs of fractional torus $\mathbf{T}^2_\a$ for $\a=\pi/2, \pi/3, \pi/6$ are shown in  Figure \ref{fig:1.1} (b). Functions on $\mathbf{T}^n_\a$ are functions $f$ on $\rn$ that satisfy $f(x+|\sin\a|m)= f(x)$ for any $x\in \rn$ and $m\in  \mathbf{Z}^n$. Such functions are called $|\sin\a|$-periodic in every coordinate for fixed $\a$.
\begin{figure}[htb]
\centering
\subfigure[Torus $\mathbf{T}^2$]{
\includegraphics[width=0.35\linewidth]{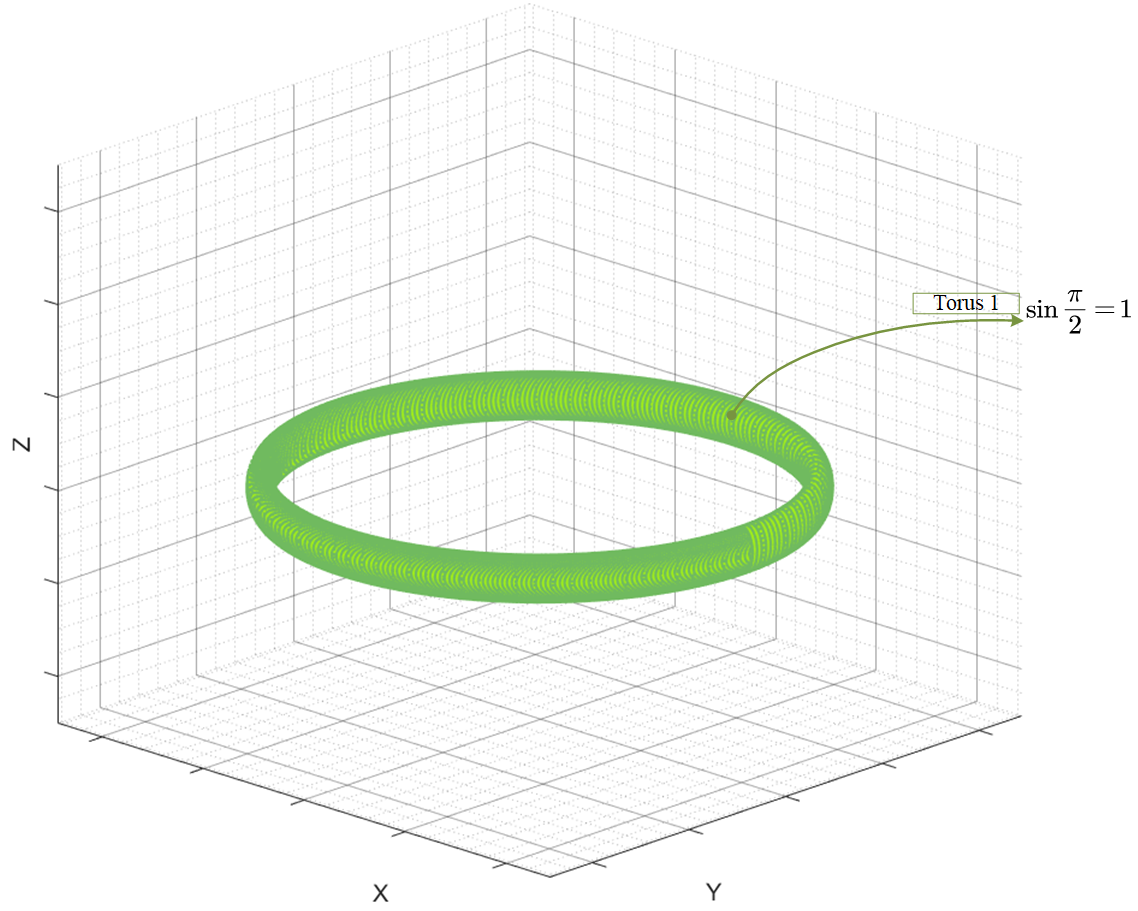}} \hspace{0.5em}
\subfigure[fractional torus of order $\a$ $\mathbf{T}^2_\a$]{
\includegraphics[width=0.35\linewidth]{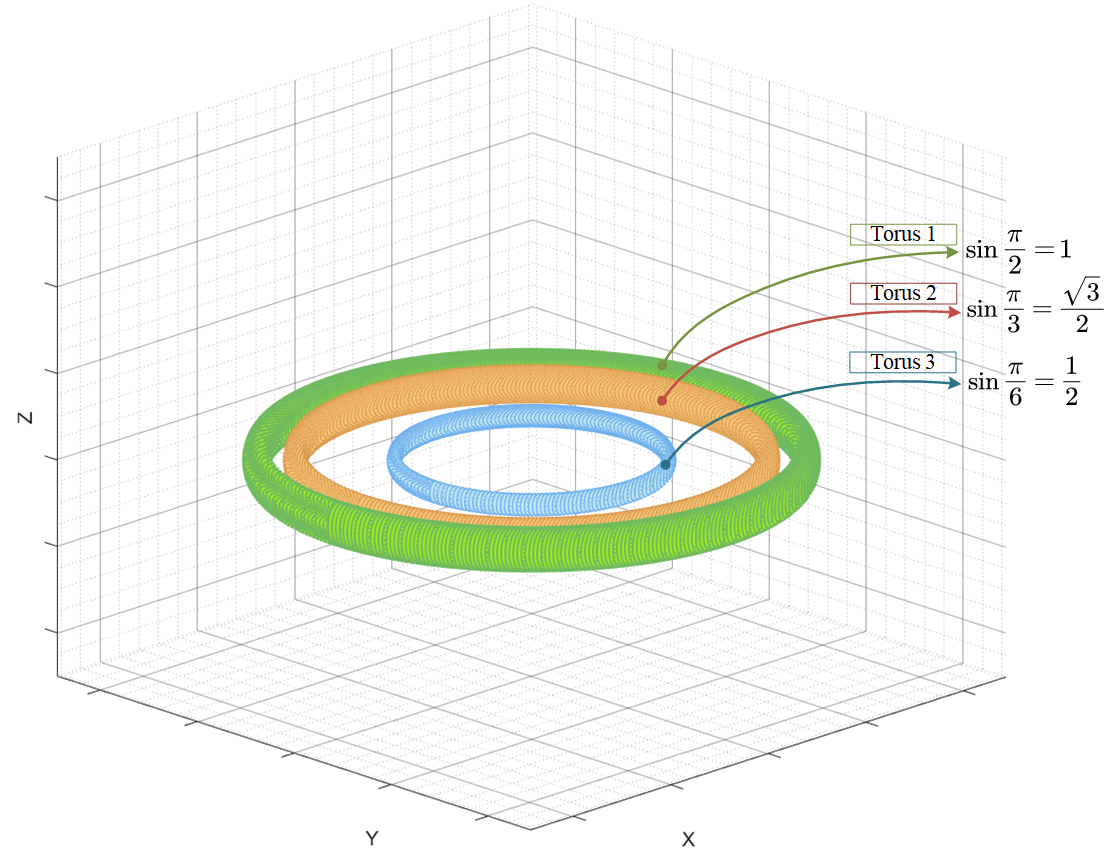}} \caption{The two-dimensional fractional torus of order $\a$ $\mathbf{T}^2_\a$}%
\label{fig:1.1}%
\end{figure}

The idea of the fractional power of the Fourier operator first appeared in the work of Wiener \cite{Wiener}. Namias \cite{Namias} introduced the FRFT to address specific types of ordinary and partial differential equations encountered in quantum mechanics. McBride and Kerr \cite{MK} provided a rigorous definition of the FRFT in integral form on the Schwartz space $S(\mathbb{R})$ based on a modification of Namias' fractional operators. Subsequently, Kerr \cite{Kerr} discussed the $L^2(\mathbb{R})$ theory of FRFT. Zayed \cite{Zayed3} introduced a novel convolution structure for the FRFT that preserves the convolution theorem for the Fourier transform. In \cite{Zayed2}, Zayed presented a new class of fractional integral transforms, encompassing the fractional Fourier and Hankel transforms. Additionally, Zayed \cite{Zayed1} explored the two-dimensional FRFT and investigated its properties, including the inversion theorem, convolution theorem, and Poisson summation formula. Kamalakkannan and Roopkumar \cite{KP} established the convolution theorem and product theorem for the multidimensional fractional Fourier transform. The $L^p(\mathbb{R})$ theory of FRFT for $1\leq p<2$ was established in \cite{CFGW}. Fu et al \cite{FGLS, FLYY} introduced the Riesz transform associated with the FRFT and explored its applications in image edge detection.
The FRFT also has various applications in many fields, such as optics \cite{OZK}, signal processing \cite{ PYL, TDZW}, image processing \cite{LRZZ, LTW, YWJK, YWJ}, and so on. In this paper, our focus centers on studying the convergence and applications of the fractional Fourier series on the torus.

This paper is organized as follows. In Section 2, we introduce fractional Fourier coefficients and give some basic properties including fractional convolution on $\mathbf{T}^n_\a$. In Section 3, we establish the fractional Fourier inversion and Poisson summation formula. Section 4 is devoted to the relationship between the decay of fractional Fourier coefficients and the smoothness of a function. The pointwise convergences of fractional Fej\'{e}r means are given in Section 5. In Section 6, using fractional Fourier series on $\mathbf{T}^n_\a$, we obtain the solutions of the fractional heat equation and fractional Dirichlet problem. Finally, we present a non-stationary signal on $\mathbf{T}^n_\a$, which can be recovered by an approximating method.

\section{Fractional Fourier series on the torus}
In this section, we begin by introducing the definition of fractional Fourier series  in the setting $\mathbf{T}^n_\a$. Subsequently, we establish some basic facts of fractional Fourier analysis. Additionally, we will give the fractional convolution and obtain fractional approximation in $L^p(\mathbf{T}^n_\a)$, $1\leq p\leq\infty$.

Let $\a\in \mathbb{R}$, $\a\neq \pi\mathbf{Z}$. Set $e_{\alpha}(x):=e^{\pi i |x|^2\cot\a}$ and $e_{\alpha}f(x):=e_{\alpha}(x)f(x)$ for a function $f$ on $\mathbf{T}^n_\a$.
\begin{definition}
Let $1\leq p\leq\infty$. We say that a function $f$ on $\mathbf{T}^n_\a$ lies in the space $e_{-\alpha}L^p(\mathbf{T}^n_\a)$ if
\begin{align*}
 f(x)=e_{-\alpha}g(x),\quad g\in L^p(\mathbf{T}^n_\a)
\end{align*}
 and satisfies $\|f\|_{L^p(\mathbf{T}^n_\a)}<\infty$.

\end{definition}

\begin{definition}\label{def1.1}\,\
For a complex-valued function $f\in e_{-\alpha}L^1(\mathbf{T}^n_\a)$, $\a\in \mathbb{R}$ and $m\in \mathbf{Z}^n$, we define
\[\mathcal{F}_\a (f)(m)=\displaystyle\left\{\begin{array}{ll}
\displaystyle\int_{\mathbf{T}^n_\a}f(x)K_\a(m,x)dx,&\a \neq \pi\mathbf{Z}, \\
f(m),\,&\a=2\pi\mathbf{Z},\\
f(-m),\,&\a=2\pi\mathbf{Z}+\pi,
\end{array}\right.\]
where
\begin{align*}
K_\a(m,x):=A_\a^n e_{\alpha}(x)e_{\alpha}(m,x)e_{\alpha}(m),
\end{align*}
here $A_\a=\sqrt{1-i\cot\a}$ and $e_{\alpha}(m,x)=e^{-2\pi i(m\cdot x)\csc \a}$.
We call $\mathcal{F}_\a (f)(m)$ the $m$-th fractional Fourier coefficient of order $\a$ of $f$.
\end{definition}

In order to state our results, we recall some notations. The spaces $C^k(\mathbf{T}^n_\a)$, $k\in \mathbf{Z}^+$, are defined as the sets of functions $\phi$ for which $\partial^\beta\phi$ exist and are continuous for all $|\beta|\leq k$. When $k = 0$ we set $C^0(\mathbf{T}^n_\a)=C(\mathbf{T}^n_\a)$ to be the space of continuous functions on $\mathbf{T}^n_\a$. Let
\begin{align*}
C^\infty(\mathbf{T}^n_\a):=\bigcap_{k=0}^{\infty}C^k(\mathbf{T}^n_\a).
\end{align*}
\begin{definition}
For  $0\leq k\leq\infty$. The space $e_{-\alpha}C^k(\mathbf{T}^n_\a)$ is defined to be the space of functions $f$ on $\mathbf{T}^n_\a$ such that
\begin{align*}
 f(x)=e_{-\alpha}g(x),\quad g\in C^k(\mathbf{T}^n_\a).
\end{align*}

\end{definition}
Notice that the spaces $e_{-\alpha}C^k(\mathbf{T}^n_\a)$ are contained in $e_{-\alpha}L^p(\mathbf{T}^n_\a)$ for all $1\leq p\leq \infty$.

We denote by $\overline{f}$ the complex conjugate of the function $f$, by $\tilde{f}$ the function $\tilde{f}(x)= f(-x)$, and by $\tau_y(f)(x)$ the function $\tau_y(f)(x)=f(x-y)$ for all $y\in \mathbf{T}^n_\a$. Since $\a = \pi\mathbf{Z}$, we get $\mathbf{T}^n_\a=\{0\}$. Hence, throughout this paper, for $\a \neq \pi\mathbf{Z}$, we always assume $\mathbf{T}^n_\a=[-|\sin\a|/2,|\sin\a|/2]^n$. Next, we give some elementary properties of fractional Fourier coefficients.
\begin{proposition}\label{pro1.2}\,\
Let $f,  g$ be in $e_{-\alpha} L^1(\mathbf{T}^n_\a)$. Then for all $m,k \in \mathbf{Z}^n$, $\lam \in \mathbb{C}$, $y\in \mathbf{T}^n_\a$, we have
\begin{enumerate}
[(1)]
\item $\mathcal{F}_\a (f+g)(m)=\mathcal{F}_\a (f)(m)+\mathcal{F}_\a (g)(m)$;
\item $\mathcal{F}_\a (\lambda f)(m)=\lambda \mathcal{F}_\a (f)(m)$;
\item $\mathcal{F}_\a ( \overline{f})(m)= \overline{{\mathcal{F}_{-\a}} (f)(m)}$;
\item $\mathcal{F}_\a ( \widetilde{f})(m)= \mathcal{F}_{\a} (f)(-m)$;
\item $\mathcal{F}_\a [e_{-\alpha}\tau_y(e_{\alpha}f)](m)=\mathcal{F}_{\a} (f)(m)e_{\a}(m,y)$;
\item $\mathcal{F}_{\a} [e_{-\alpha}(k,\cdot)f](m)e^{-2\pi i(m\cdot k)\cot\a}e_\a({k})= \mathcal{F}_\a (f)(m-k)$;
\item $\mathcal{F}_\a (f)(0)=A_\a^n \displaystyle\int_{\mathbf{T}^n_\a}e_{-\alpha}f(x)dx$;
\item $\sup_{m\in \mathbf{Z}^n}|\mathcal{F}_\a (f)(m)|\leq|\csc\a|^{n/2}\|f\|_{L^1(\mathbf{T}^n_\a)}$;
\item $\mathcal{F}_\a [e_{-\alpha}\partial^\beta(e_{\alpha}f)](m)=(2\pi i m\csc\a)^\beta\mathcal{F}_\a (f)(m),$\; whenever $f\in e_{-\a} C^\beta(\mathbf{T}^n_\a)$.
\end{enumerate}
\end{proposition}
\begin{proof}
It is obvious that properties (1)-(4) and (7) hold. We now pay attention to (5). Note that
\begin{align*}
\mathcal{F}_\a [e_{-\alpha}\tau_y\big(e_{\alpha}f\big)](m)
=&A_{\a}^ne_{\alpha}(m)\int_{\mathbf{T}^n_\a}(e_{\alpha}f)(x-y)e_{\alpha}(m,x)dx\\
=&A_{\a}^ne_{\alpha}(m)\int_{\mathbf{T}^n_\a-y}(e_{\alpha}f)(x')e_{\alpha}(m,x'+y)dx'\\
=&\mathcal{F}_{\a} (f)(m)e_{\a}(m,y),
\end{align*}
where we make the variable change $x'=x-y$ in the second equality, and the third equality follows from the periodicity of function.

Next, we deal with (6).
\begin{align*}
\mathcal{F}_\a (f)(m-k)
=&A_{\a}^n\int_{\mathbf{T}^n_\a}f(x)e_{\a}(x)e_{\a}(m-k)e_{\a}(m-k,x)dx\\
=&e_\a(k)e^{-2\pi i(m\cdot k)\cot \a}\int_{\mathbf{T}^n_\a}e_{-\a}(k,x)f(x)K_{\a}(m,x)dx\\
=&\mathcal{F}_{\a} [e_{-\a}(k,\cdot)f](m)e^{-2\pi i(m\cdot k)\cot\a}e_\a(k).
\end{align*}
It follows from the fact $|K_\a(m,x)|\leq |\csc\a|^{n/2}$ that (8) holds.

Now we turn to (9). By the periodicity and integration by parts, we have
\begin{align*}
\mathcal{F}_\a [e_{-\alpha}\partial^\beta(e_{\alpha}f)](m)
=&A_\a^n e_\a(m)\int_{\mathbf{T}^n_\a}\partial^\beta(e_{\a} f)(x)e^{-2\pi i (m\cdot x)\csc \a}dx\\
=&(2\pi i m\csc\a)^\beta A_\a^n e_\a(m)\int_{\mathbf{T}^n_\a} e_{\a}f(x) e^{-2\pi i (m\cdot x)\csc \a}dx\\
=&(2\pi i m\csc\a)^\beta\mathcal{F}_\a(f)(m).
\end{align*}
This completes the proof.
\end{proof}
\begin{remark}\label{re2.7}
Suppose $f_1\in e_{-\alpha} L^1(\mathbf{T}^{n_1}_\a)$ and  $f_2\in e_{-\alpha} L^1(\mathbf{T}^{n_2}_\a)$. We obtain that the tensor function
$$(f_1\otimes f_2)(x_1,x_2)=f_1(x_1)f_2(x_2)$$
is in $e_{-\alpha} L^1(\mathbf{T}^{n_1+n_2}_\a)$. Moreover, it is easy to check
\begin{align}\label{eq2.11}
\mathcal{F}_\a\big(f_1\otimes f_2\big)(m_1,m_2)=\mathcal{F}_\a(f_1)(m_1)\mathcal{F}_\a(f_2)(m_2),
\end{align}
where $m_1\in \mathbf{Z}^{n_1}$ and $m_2\in \mathbf{Z}^{n_2}$.
\end{remark}

\begin{definition}\label{def1.2}\,\
For $\a \neq n\pi$, a trigonometric polynomials of order $\a$ on $\mathbf{T}^{n}_\a$ is a function of the form
\begin{align*}
P_\a(x)=\sum_{m\in \mathbf{Z}^n}c_{m,\a}K_{-\a}(m,x),
\end{align*}
where $m =(m_1,\ldots,m_n)$ and $c_{m,\a}$ is a constant depending on $m$ and $\a$. The degree of $P_\a$ is the largest number $|m_1|+\cdots+|m_n|$ such that $c_{m,\a}$ is nonzero. Observe that
\begin{align*}
\mathcal{F}_\a (P_\a)(m)=c_{m,\a}.
\end{align*}
\end{definition}
For $x\in\mathbf{T}^n_\a$, $f\in e_{-\alpha} L^1(\mathbf{T}^{n}_\a)$, the fractional Fourier series of order $\a$ of $f$ is the series
\begin{align}\label{eq2.1}
\sum_{m\in \mathbf{Z}^n}\mathcal{F}_\a (f)(m)K_{-\a}(m,x).
\end{align}
Now, we wonder (\ref{eq2.1}) converges in which sense? The convergence of the fractional Fourier series is the main topic in this paper. Next, we introduce a kind of the fractional convolution of order $\a$ to study the convergence of fractional Fourier series.
\begin{definition}\label{def1}\,\
Let $f\in e_{-\alpha} L^1(\mathbf{T}^{n}_\a)$ and $g\in L^1(\mathbf{T}^{n})$. Define the fractional convolution of order $\a$ as follows:
\begin{align*}
 (f\overset{\alpha}{\ast}g)(x)=|\csc \a|^n{e_{-\a}(x)\int_{\mathbf{T}^n_\a}e_{\a}(y)f(y)(\delta^\alpha g)(x-y)dy},
\end{align*}
where $(\delta^\alpha g)(x)=g(x\csc\a)$.
\end{definition}

Let $f\in e_{-\a}L^1(\mathbf{T}^n_\a)$. We have
\begin{align*}
\sum_{|m|\leq N}\mathcal{F}_\a (f)(m)K_{-\a}(m,x)&=\sum_{|m|\leq N}\int_{\mathbf{T}^n_\a} f(y)K_\a(m,y)dyK_{-\a}(m,x)\\
&=|\csc \a|^ne_{-\a}(x)\sum_{|m|\leq N}\int_{\mathbf{T}^n_\a} e_{\a}(y)f(y)e^{2\pi i m \cdot (x-y)\csc \a}dy\\
&=|\csc \a|^ne_{-\a}(x)\int_{\mathbf{T}^n_\a}e_{\a}(y)f(y)\sum_{|m|\leq N}e^{2\pi i m \cdot (x-y)\csc \a}dy\\
&=:\big(f\overset{\alpha}\ast {D}_N^n \big)(x),
\end{align*}
where ${D}^n_N(x)$ is the classical multidimensional Dirichlet kernel.

Noting that one-dimensional Dirichlet kernel ${{D}}_N^1$ is not an approximate identity. This fact led Ces\`{a}ro and independently Fej\'{e}r to consider the arithmetic means of the Dirichlet kernel in dimension $1$ as follows
\begin{align*}
{F}^1_N(x):=\frac{1}{N+1}\big[{D}_0^1(x)+{D}_1^1(x)+\cdots+{D}_N^1(x)\big].
\end{align*}
By Proposition 3.1.7 in \cite{G}, we have the following equivalent way to write the kernel ${F}^1_N$:
\begin{align}\label{eq2.33}
{{F}}^1_N(x)=\sum_{j=-N}^N\Big(1-\frac{|j|}{N+1}\Big)e^{2\pi i j x}={\frac{1}{N+1}\Big(\frac{\sin\big((N+1)\pi x \big)} {\sin(\pi x)}\Big)^2}.
\end{align}
The function ${{F}}^1_N$ given by (\ref{eq2.33}) is called the Fej\'{e}r kernel.
Let $${{F}}^{1,\a}_N(x):=|\csc\a| F^1_N(x\csc\a).$$
We have the following identity for the kernel ${{F}}^{1,\a}_N$:
\begin{proposition}\label{pro1.3}\,\
For every nonnegative integer $N$, the following identity holds
\begin{align}\label{eq2.3}
{{F}}^{1,\a}_N(x)=|\csc \a|\sum_{j=-N}^N\Big(1-\frac{|j|}{N+1}\Big)e^{2\pi i j x\csc \a}={\frac{|\csc \a|}{N+1}\Big(\frac{\sin\big((N+1)\pi x \csc\a\big)} {\sin(\pi x \csc\a)}\Big)^2}
\end{align}
for all $x\in \mathbf{T}^1_\a$. Thus
$$\mathcal{F}_\a\big(e_{-\a}{{F}}^{1,\a}_N\big)(m)=A_\a e_{\a}(m)\Big(1-\frac{|m|}{N+1}\Big)$$ if $|m|\leq N$ and zero otherwise.
\end{proposition}
\begin{proof}
Note that ${{F}}^{1,\a}_N\in L^1(\mathbf{T}^1_\a)$. We have
\begin{align}\label{eq1.4}
\nonumber\mathcal{F}_\a\big(e_{-\a}{{F}}^{1,\a}_N\big)(m)
=&\int_{\mathbf{T}^1_\a}e_{-\a}(x){{F}}^{1,\a}_N(x)K_\a(m,x)dx\\
\nonumber=&A_\a|\csc \a|e_{\a}(m)\sum_{j=-N}^N\Big(1-\frac{|j|}{N+1}\Big)\int_{\mathbf{T}^1_\a}e_{\a}(m,x)e_{-\a}(j,x)dx\\
=&A_\a e_{\a}(m)\Big(1-\frac{|m|}{N+1}\Big),
\end{align}
where
\begin{align*}
|\csc \a|\sum_{j=-N}^N\Big(1-\frac{|j|}{N+1}\Big)\int_{\mathbf{T}^1_\a}e_{\a}(m,x)e_{-\a}(j,x)dx=0,\qquad m\neq j,
\end{align*}
\begin{align*}
|\csc \a|\sum_{j=-N}^N\Big(1-\frac{|j|}{N+1}\Big)\int_{\mathbf{T}^1_\a}e_{\a}(m,x)e_{-\a}(j,x)dx=1-\frac{|m|}{N+1},\qquad m=j.
\end{align*}
This completes the proof.
\end{proof}
\begin{definition}\label{def1.6}\,\
Let $N$ be a nonnegative integer. The function ${{F}}^{1,\a}_N$ on $\mathbf{T}^1_\a$ given by $(\ref{eq2.3})$ is called the fractional Fej\'{e}r kernel of order $\a$.
\end{definition}
The Fej\'{e}r kernel ${{F}}^{n}_N$ on the torus is defined as the product of the $1$-dimensional Fej\'{e}r kernels. Similarly, we define fractional Fej\'{e}r kernel ${{F}}^{n,\a}_N$ on $\mathbf{T}^n_\a$ as the product of the $1$-dimensional fractional Fej\'{e}r kernels, namely
\begin{align*}
{{F}}^{n,\a}_N(x_1,\ldots,x_n):=&\prod_{j=1}^n{{F}}^{1,\a}_N(x_j).
\end{align*}
\begin{remark}\label{rem1.3}\,\
For all $N\geq 0$, by $(\ref{eq2.3})$, we get
\begin{align*}
{{F}}^{n,\a}_N(x_1,\ldots,x_n)=& |\csc \a|^n\sum_{m\in \mathbf{Z}^n, |m_j|\leq N}\prod_{j=1}^n\Big(1-\frac{|m_j|}{N+1}\Big) e^{2\pi i (m\cdot x)\csc \a}\\
=&\frac{|\csc \a|^n}{(N+1)^n}\prod_{j=1}^n\Big(\frac{\sin\big((N+1)\pi x_j \csc\a\big)} {\sin(\pi x_j \csc\a)}\Big)^2.
\end{align*}
Thus ${{F}}^{n,\a}_N\geq 0$. Note that ${{F}}^{n,\a}_0(x)=|\csc \a|^n$ and  ${{F}}^{n,\a}_N(0)=|\csc \a|^n(N+1)^n$.
\end{remark}

\begin{proposition}\label{pro2.4}\,\
The family of fractional Fej\'{e}r kernels $\{{{F}}^{n,\a}_N\}_{N=0}^\infty$ is an approximate identity on $\mathbf{T}^n_\a$.
\end{proposition}
\begin{proof}
It is easy to see that
\begin{align*}
\|{F}^{n,\a}_N\|_{L^1(\mathbf{T}^n_\a)}=\int_{\mathbf{T}^n_\a}{F}^{n,\a}_N(x)dx
=\prod_{j=1}^N\int_{\mathbf{T}^1_\a}{F}^{1,\a}_N(x_j)dx_j=1.
\end{align*}
By noting that $1\leq {|t|}{/|\sin t|}\leq {\pi}{/2}$ when $|t|\leq {\pi}{/2}$, we obtain
\begin{align*}
{F}^{1,\a}_N(x)\leq& \frac{|\csc\a|}{N+1}\min \Big(\frac{(N+1)|\pi x \csc\a|}{|\sin(\pi x \csc\a)|},
\frac{1}{|\sin(\pi x \csc\a)|}\Big)^2\\
\leq&\frac{|\csc\a|}{N+1}\frac{\pi^2}{4}\min \Big(N+1,\frac{1}{|\pi x \csc\a|}\Big)^2,
\end{align*}
where $|x|\leq\frac{1}{2|\csc\a|}$. For $\delta>0$, we get
\begin{align*}
\int_{\delta\leq |x|\leq \frac{1}{2|\csc\a|}}{F}^{1,\a}_N(x)dx\leq \frac{|\csc\a|}{N+1}\frac{\pi^2}{4}
\int_{\delta\leq |x|\leq \frac{1}{2|\csc\a|}}\frac{1}{|\pi \delta \csc\a|^2}dx
\leq\frac{1}{4\delta^2}\frac{1}{N+1}\rightarrow 0
\end{align*}
as $N\rightarrow \infty$.

In higher dimensions, given $x =(x_1,\ldots,x_n)\in
[-{|\sin\a|}/{2},{|\sin\a|}/{2}]^n$ with $|x|\geq \delta,$ there exists a $j \in \{1,\ldots,n\}$ such that
$|x_j|\geq \delta/\sqrt n$ and thus
\begin{align}\label{eq1.5}
\int_{|x|\geq \delta}{F}^{n,\a}_N(x)dx
\leq& \int_{|x_j|\geq \frac{\delta}{\sqrt n}}{F}^{1,\a}_N(x_j)dx_j\prod_{k\neq j}\int_{\mathbf{T}^1_\a}{F}^{1,\a}_N(x_k)dx_k
\leq\frac{n}{4\delta^2}\frac{1}{N+1}\rightarrow 0,
\end{align}
which completes the proof.
\end{proof}

\begin{theorem}\label{th2.5}
Let $\a\in \mathbb{R}$ and $\a\neq \pi\mathbf{Z}$.\\
 (1) If $f\in e_{-\a}L^p(\mathbf{T}^n_\a)$, $1\leq p<\infty$, then
\begin{align*}
\lim_{N\rightarrow\infty}\Big\| f\overset{\alpha}{\ast}{{F}}^n_N - f\Big\| _{L^p(\mathbf{T}^n_\a)}=0.
\end{align*}
 (2) If $f\in e_{-\a}L^\infty(\mathbf{T}^n_\a)$ is uniformly continuous on a subset $K$ of $\mathbf{T}^n_\a$, then
\begin{align*}
\lim_{N\rightarrow\infty}\Big\|f\overset{\alpha}{\ast}{{F}}^n_N - f\Big\| _{L^\infty(K)}=0.
\end{align*}.
\end{theorem}
\begin{proof}
We first deal with the case $1\leq p<\infty$. It is easy to check that
\begin{align*}
( f\overset{\alpha}{\ast}{{F}}^n_N)(x)-f(x)
=&e_{-\a}(x)\int_{\mathbf{T}^n_\a} e_{\alpha}f(t){{F}}^{n,\a}_N(x-t)\mathrm{d}t-\int_{\mathbf{T}^n_\a}f(x){{F}}^{n,\a}_N(t)\mathrm{d}t\\
=&e_{-\a}(x)\int_{\mathbf{T}^n_\a}\big((e_{\alpha}f)(x-t)-(e_{\alpha}f)(x)\big){{F}}^{n,\a}_N(t)
\mathrm{d}t.
\end{align*}
Now approximate a given $e_{\alpha}f\in L^p(\mathbf{T}^n_\a)$ by a continuous function with compact support to deduce that
\begin{align*}
\int_{\mathbf{T}^n_\a}\big|(e_{\alpha}f)(x-h)-(e_{\alpha}f)(x)\big|^pdx\rightarrow0 \quad as \quad  h\rightarrow0.
\end{align*}
For given $\varepsilon >0$, there exists $\delta>0$ such that
\begin{align}\label{eq2.2}
\int_{\mathbf{T}^n_\a}\big|(e_{\alpha}f)(x-h)-(e_{\alpha}f)(x)\big|^pdx<\frac{\varepsilon^p}{2^p}, \quad h\in [-\delta,\delta].
\end{align}
Then
\begin{align*}
( f\overset{\alpha}{\ast}{{F}}^n_N)(x)-f(x)
&=e_{\alpha}(x)\int_{|t|\leq\d}\big((e_{\alpha}f)(x-t)-(e_{\alpha}f)(x)\big){{F}}^{n,\a}_N(t)
\mathrm{d}t\\
&\qquad+e_{\alpha}(x)\int_{|t|>\d}\big((e_{\alpha}f)(x-t)-(e_{\alpha}f)(x)\big){{F}}^{n,\a}_N(t)dt\\
&=:I_1+I_2.
\end{align*}
By Minkowski's integral inequality and (\ref{eq1.5}), we obtain
\begin{align}\label{eq2.7}
\|I_2\| _{p}
&\leq\int_{|t|>\d} \Big(\int_{\mathbf{T}^n_\a}\big|\big(e_{\alpha}f\big)(x-t)
-\big(e_{\alpha}f\big)(x)\big|^pdx\Big)^{1/p}{{F}}^{n,\a}_N(t)dt\\
\nonumber&\leq2\|f\|_p\int_{|t|>\d}{{F}}^{n,\a}_N(t)dt\rightarrow0 \qquad as\quad N\rightarrow\infty.
\end{align}
In addition, by $(\ref{eq2.2})$, we get
\begin{align*}
\|I_1\| _{p} &\leq\int_{|t|\leq\d} \Big(\int_{\mathbf{T}^n_\a}\big|\big(e_{\alpha}f\big)(x-t)-\big(e_{\alpha}f\big)(x)\big|^pdx\Big)^{1/p}{{F}}^{n,\a}_N(t)dt< \frac{\varepsilon}{2}.
\end{align*}
This together with (\ref{eq2.7}) implies the required conclusion.

Now, we turn to conclusion $(2)$. Let $e_{\alpha}f$ be a bounded function on $\mathbf{T}^n_\a$ that is
uniformly continuous on K. Given $\v> 0$, there exists a neighborhood $V$ of $0$ such that
\begin{align*}
\big|(e_{\alpha}f)(x-h)-(e_{\alpha}f)(x)\big|<\frac{\v}{2} {\rm{\;for\; all\;}} h\in V\; {\rm{and}} \; x\in K.
\end{align*}
Applying this along with $(\ref{eq1.5})$, we get that as $N\to\infty$
\begin{align*}
&\sup_{x\in K}\big| ( f\overset{\alpha}{\ast}{{F}}^{n,\a}_N)(x)-f(x)\big|\\
&\leq \int_{|t|\leq\d}\sup_{x\in K} \big|(e_{\alpha}f)(x-t)-(e_{\alpha}f)(x)\big|{{F}}^{n,\a}_N(t)
\mathrm{d}t\\
&\qquad+\int_{|t|>\d}\sup_{x\in K} \big|(e_{\alpha}f)(x-t)-(e_{\alpha}f)(x)\big|{{F}}^{n,\a}_N(t)dt\\
&\leq \frac{\v}{2},
\end{align*}
which completes the proof.
\end{proof}
\begin{proposition}\label{pro1.5}\,\
The set of trigonometric polynomials of order $\a$ is dense in $e_{-\alpha}L^p(\mathbf{T}^n_\a)$ for $1 \leq p <\infty$.
\begin{proof}
For $1\leq p <\infty$, we claim that $f\overset{\alpha}{\ast}{{F}}^n_N$ is also a trigonometric polynomial of order $\a$.
In fact,
\begin{align*}
( f\overset{\alpha}{\ast}{{F}}^n_N)(x)
=&\underset{|m_j|\leq N}{\sum_{ m\in \mathbf{Z}^n}}\prod_{j=1}^n
\Big(1-\frac{|m_j|}{N+1}\Big)\mathcal{F}_\a (f)(m)K_{-\a}(m,x).
\end{align*}
It follows from Theorem \ref{th2.5} that $f\overset{\alpha}{\ast}{\mathcal{F}}^n_N$ converges to $f$ as $N\to\infty$. This finishes the proof.
\end{proof}
\end{proposition}
\begin{corollary}(Weierstrass approximation theorem for trigonometric polynomials of order $\a$) Every continuous function on $\mathbf{T}^n_\a$ is a uniform limit of trigonometric polynomials with order $\a$.
\end{corollary}
\begin{proof}
Note that $\mathbf{T}^n_\a$ is a compact set. By Theorem \ref{th2.5} (2), we know that $f\overset{\alpha}{\ast}{\mathcal{F}}^n_N$ converges uniformly to $f$  as $N\to\infty$. Since $f\overset{\alpha}{\ast}{\mathcal{F}}^n_N$ is a trigonometric polynomial of order $\a$, we conclude that every continuous function on $\mathbf{T}^n_\a$ can be uniformly approximated by trigonometric polynomials of order $\a$.
\end{proof}

\section{Reproduction of Functions from Their Fractional Fourier Coefficients }
In this section, we establish the fractional Fourier inversion on $e_{-\a}L^1(\mathbf{T}^n_\a)$ and study basic properties of $e_{-\a}L^2(\mathbf{T}^n_\a)$. Moreover, to explore the connections between fractional Fourier analysis on $\mathbf{T}^n_\a$ and fractional Fourier analysis on $\rn$, we give the fractional Poisson summation formula.
\subsection{Fractional Fourier inversion }\label{section3.1}
We now define the partial sums of fractional Fourier series.
\begin{definition}\label{def3.1}\,\
Let $N\in \mathbb{N}$, $\a\in \mathbb{R}$ and $\a\neq \pi\mathbf{Z}$. The expressions
\begin{align}\label{eq2.9}
(f\overset{\alpha}{\ast}{F}_N^n)(x)
=&\underset{|m_j|\leq N}{\sum_{ m\in \mathbf{Z}^n}}\prod_{j=1}^n
\Big(1-\frac{|m_j|}{N+1}\Big)\mathcal{F}_\a(f)(m)K_{-\a}(m,x)
\end{align}
are called the fractional Fej\'{e}r means (or fractional Ces\`{a}ro means) of $f$.
\end{definition}
In the following propositions, we obtain that the fractional Fourier series uniquely determine the function.
\begin{proposition}\label{pro2.1}\,\
If $ f, g \in e_{-\a}L^1(\mathbf{T}^n_\a)$ satisfy $\mathcal{F}_\a (f)(m)=\mathcal{F}_\a (g)(m)$ for all $m\in \mathbf{Z}^n$, then $f=g$ a.e.
\end{proposition}
\begin{proof}
By linearity of $\mathcal{F}_\a$, we can set $g=0$. If $\mathcal{F}_\a(f)(m)=0$ for all  $m\in \mathbf{Z}^n$, we know from (\ref{eq2.9}) that $f\overset{\alpha}{\ast}{F}_N^n =0$ for all $N\in \mathbf{Z}^+$. By Theorem \ref{th2.5}, we have
\begin{align*}
\big\|f\overset{\alpha}{\ast}{{F}}^n_N-f\big\|_{L^1(\mathbf{T}^n_\a)}\rightarrow 0
\end{align*}
as $N\rightarrow \infty$. Then
\begin{align*}
\|f\|_{L^1(\mathbf{T}^n_\a)}\leq\big\Vert f\overset{\alpha}{\ast}{{F}}^n_N - f\big\Vert _{L^1(\mathbf{T}^n_\a)}+\big\| f\overset{\alpha}{\ast}{{F}}^n_N\big\|_{L^1(\mathbf{T}^n_\a)}\rightarrow 0,
\end{align*}
from which we conclude that $f = 0$ a.e.
\end{proof}
Whether the partial sums of the fractional Fourier series converge to the function as $N\rightarrow \infty$? We give a firm answer to this question.
\begin{corollary}\label{cor3.3}\,\
(Fractional Fourier inversion) Suppose $ f\in e_{-\a}L^1(\mathbf{T}^n_\a)$ and
$$\sum_{ m\in \mathbf{Z}^n}|\mathcal{F}_\a (f)(m)|<\infty.$$
Then
\begin{align}\label{eq2.5}
f(x)=\sum_{ m\in \mathbf{Z}^n}\mathcal{F}_\a(f)(m) K_{-\a}(m,x) \quad a.e.,
\end{align}
and therefore $f$ is almost everywhere equal to a continuous function.
\end{corollary}
\begin{proof}
For $m\neq n$, it is easy to see that
\begin{align*}
K_{-\a}(m,x) K_{\a}(n,x)
=|\csc\a|^ne_{-\a}(m)e_{\a}(n)e_{\a}(n-m,x).
\end{align*}
This implies
\begin{align*}
\int_{\mathbf{T}^n_\a}K_{-\a}(m,x) K_{\a}(n,x)dx
=0,\quad m\neq n.
\end{align*}
Meanwhile
\begin{align*}
\int_{\mathbf{T}^n_\a}K_{-\a}(m,x) K_{\a}(m,x)dx
=&|\csc\a|^n\int_{\mathbf{T}^n_\a}dx=1.
\end{align*}
Set
$$G(x):=\sum_{ j\in \mathbf{Z}^n}\mathcal{F}_\a(f)(j) K_{-\a}(j,x).$$
It is obvious that $G\in e_{-\a}L^1(\mathbf{T}^n_\a)$. Consequently,
\begin{align*}
\mathcal{F}_\a(G)(m)
=&\int_{\mathbf{T}^n_\a}\Big(\sum_{ j\in \mathbf{Z}^n}\mathcal{F}_\a(f)(j) K_{-\a}(j,x)\Big)K_{\a}(m,x)dx=\mathcal{F}_\a (f)(m).
\end{align*}
By Proposition \ref{pro2.1}, we have
\begin{align*}
f(x)=\sum_{ m\in \mathbf{Z}^n}\mathcal{F}_\a(f)(m) K_{-\a}(m,x) \quad a.e.
\end{align*}
This completes the proof.
\end{proof}

\subsection{Fractional Fourier series of square summable functions}\label{section3.2}
Now consider the space $e_{-\a}L^2(\mathbf{T}^n_\a)$ with inner product
\begin{align*}
\langle f|g  \rangle=\int_{\mathbf{T}^n_\a}f(t)\overline{g(t)}dt.
\end{align*}
It is easy to check that
\[\int_{\mathbf{T}^n_\a}K_{-\a}(m,x) \overline{K_{-\a}(n,x)}dx=\left\{\begin{array}{ll}
1,&when\; m=n,\\
0,\,&when \; m\neq n.
\end{array}\right.\]
This implies that the sequence $\{K_{-\a}(m,\cdot)\}$ is orthonormal. Meanwhile, for all $f\in e_{-\a}L^2(\mathbf{T}^n_\a)$, we have
\begin{align*}
\langle f|K_{-\a}(m,\cdot)\rangle
=&\int_{\mathbf{T}^n_\a}f(y)\overline{K_{-\a}(m,y)}dy=\mathcal{F}_\a(f)(m).
\end{align*}
If  $\langle f|K_{-\a}(m,\cdot)\rangle=0$ for all $m\in \mathbf{Z}^n$, we know from Proposition \ref{pro2.1} that $f=0$ a.e. Therefore, the completeness of the sequence $\{K_{-\a}(m,\cdot)\}$ holds.
\begin{proposition}\label{pro3.5}\,\
The following are valid for {$f,g \in e_{-\a}L^2(\mathbf{T}^n_\a)$} and\\
(1) (Plancherel's identity)
\begin{align*}
\|f\|_2^2=\sum_{m\in \mathbf{Z}^n}|\mathcal{F}_\a(f)(m)|^2.
\end{align*}
(2) The function $f(x)$ is a.e. equal to the $e_{-\a}L^2(\mathbf{T}^n_\a)$ limit of the sequence
\begin{align*}
\sum_{|m|\leq M }\mathcal{F}_\a (f)(m)K_{-\a}(m,x).
\end{align*}
(3) (Parseval's relation)
\begin{align*}
\int_{\mathbf{T}^n_\a}f(t)\overline{g(t)}dt=\sum_{m\in \mathbf{Z}^n}\mathcal{F}_\a(f)(m)\overline{\mathcal{F}_\a(g)(m)}.
\end{align*}
(4) The map $f\mapsto\big\{\mathcal{F}_\a(f)(m)\big\}$ is an isometry from $e_{-\a}L^2(\mathbf{T}^n_\a)$ onto $l^2$.\\
(5) For all $k\in \mathbf{Z}^n$, we have
\begin{align*}
\mathcal{F}_\a\big[e_{\a}(\cdot)fg\big](m)={A_{\a}^n}e_{\a}^2(m)\sum_{j\in \mathbf{Z}^n}\mathcal{F}_\a (f)(j)e^{-2\pi i(j\cdot m)\cot\a}{\mathcal{F}_{-\a}(e_{\a}^2(\cdot)g)(j-m)}.
\end{align*}
\end{proposition}
\begin{proof}
The proofs of (1), (2), (3) and (4) are quite similar to those in the case of Fourier series, so we omit them.
{Now we turn to (5). Let $\overline{G(x)}=e_{\a}(x)g(x)K_{\a}(m,x)$}. From (3), we have
\begin{align*}
\mathcal{F}_\a \big[e_{\a}(fg)\big](m)=&\int_{\mathbf{T}^n_\a}e_{\a}(x)f(x)g(x) K_{\a}(m,x)dx\\
=&\int_{\mathbf{T}^n_\a}f(x)\overline{G(x)}dx\\
=&\sum_{j\in \mathbf{Z}^n}\mathcal{F}_\a(f)(j)\overline{\mathcal{F}_\a(G)(j)}.
\end{align*}
In view of Proposition \ref{pro1.2} (6), we get
\begin{align*}
{\mathcal{F}_\a(G)(j)}
=&\int_{\mathbf{T}^n_\a}\overline{e_{\a}(x)g(x)K_{\a}(m,x)}K_{\a}(j,x)dx\\
=&\overline{A_{\a}^n}e_{-\a}(m)\int_{\mathbf{T}^n_\a}\big[\overline{e^2_{\a}(x)g(x)}e^{2\pi imx\csc\a}\big]K_{\a}(j,x)dx\\
=&\overline{A_{\a}^n}e_{-\a}(m)\mathcal{F}_\a\big(\overline{e_{\a}^2g}e_{-\a}(m,\cdot)\big)(j)\\
=&\overline{A_{\a}^n}e_{-\a}^2(m)e^{2\pi i(j\cdot m)\cot\a}\mathcal{F}_\a\big(\overline{e^2_{\a}g}\big)(j-m).
\end{align*}
By Proposition \ref{pro1.2} (3), we have
\begin{align*}
\overline{\mathcal{F}_\a(G)(j)}
=&{A_{\a}^n}e^2_{\a}(m)e^{-2\pi i(j\cdot m)\cot\a}\overline{\mathcal{F}_\a\big(\overline{e^2_{\a}g}\big)(j-m)}\\
=&{A_{\a}^n}e^2_{\a}(m)e^{-2\pi i(j\cdot m)\cot\a}{\mathcal{F}_{-\a}\big(e^2_{\a}g\big)(j-m)}.
\end{align*}
\end{proof}
\subsection{The fractional Poisson summation formula }\label{section3.3}
In this subsection, we establish an important connection between fractional Fourier analysis on $\mathbf{T}^n_\a$ and Fourier analysis on $\rn$.
\begin{theorem}\label{th3.1}
{\rm(Fractional Poisson summation formula)} Suppose that $f$ is a continuous function on $\rn$. If there exist $C,\delta>0$ such that
\begin{align}\label{eq3.8}
|f(x)|\leq \frac{C}{(1+|x|)^{n+\delta}},\quad \forall x\in \rn
\end{align}
and the fractional Fourier transform $\mathcal{F}_\a(e_{-\a}f)$ restricted on $\mathbf{Z}^n$ satisfies
\begin{align}\label{eq4.1}
\sum_{ m\in \mathbf{Z}^n}\big|\mathcal{F}_\a\big(e_{-\a}f\big)(m)\big|<\infty.
\end{align}
Then for all $x\in\rn$ we have
\begin{align}\label{eq4.2}
\sum_{ m\in \mathbf{Z}^n}\mathcal{F}_\a\big(e_{-\a}f\big)(m)K_{-\a} (m,x)=e_{-\a}(x)\sum_{ k\in \mathbf{Z}^n} f(x+k|\sin\a|),
\end{align}
and in particular
\begin{align*}
A_{-\a}^n\sum_{ m\in \mathbf{Z}^n}\mathcal{F}_\a\big(e_{-\a}f\big)(m)e_{-\a}(m)=\sum_{ k\in\mathbf{ Z}^n}f(k|\sin\a|).
\end{align*}
\end{theorem}
\begin{proof}
Suppose
$$F(x)=e_{-\a}(x)\sum_{ k\in \mathbf{Z}^n}f(x+k|\sin\a|).$$
It is easy to check that $F\in e_{-\a}L^1(\mathbf{T}^n_\a)$. We show that the sequence of the fractional Fourier coefficients of $F$ coincides with the restriction of the fractional Fourier transform of $f$ on $\mathbf{Z}^n$. In fact, for any $k\in\mathbf{Z}^n$, we get
\begin{align*}
|k|\sin\a|+x|\geq |k\sin\a|-|x|\geq |k\sin\a|-\frac{|\sin\a|}{2}\sqrt n
\end{align*}
for all  $x\in [-|\sin\a|/2,|\sin\a|/2]^n$. This together with (\ref{eq3.8}) implies
\begin{align}\label{eq3.9}
\sum_{ k\in \mathbf{Z}^n} \big|f(x+k|\sin\a|)\big|
\leq& \sum_{ k\in \mathbf{Z}^n}\frac{C_{n,\delta}}{(1+|k\sin\a|)^{n+\delta}}<\infty.
\end{align}
Therefore,
\begin{align*}
\mathcal{F}_\a(F)(m)
=&A_\a^n e_{\a}(m)\int_{\mathbf{T}^n_\a}e^{-2\pi i (m \cdot x)\csc \a}\sum_{ k\in \mathbf{Z}^n}f(x+k|\sin\a|)dx\\
=&A_\a^n e_{\a}(m)\sum_{ k\in\mathbf{ Z}^n}\int_{\mathbf{T}^n_\a}e^{-2\pi i (m \cdot x)\csc \a} f(x+k|\sin\a|)dx\\
=&A_\a^n e_{\a}(m)\sum_{ k\in \mathbf{Z}^n}\int_{\mathbf{T}^n_\a+k|\sin\a|}e^{-2\pi i (m \cdot x)\csc \a}f(x)dx\\
=&A_\a^n e_{\a}(m)\int_{\mathbb{R}^n}e^{-2\pi i (m \cdot x)\csc \a}f(x)dx\\
=&\mathcal{F}_\a(e_{-\a}f)(m),
\end{align*}
where the second equality follows from (\ref{eq3.9}). Meanwhile, we also obtain that $F$ is continuous. It follows that (\ref{eq4.1}) holds with $|\mathcal{F}_\a(F)|$ in place of $|\mathcal{F}_\a(e_{-\a}f)|$. By Corollary \ref{cor3.3}, we obtain that (\ref{eq4.2}) holds for $x\in \mathbf{T}^n_\a$ and then by periodicity, for all $x\in \rn$.
\end{proof}

\section{Decay of Fractional Fourier coefficient}
In this section, we study the decay of the fractional Fourier coefficients associated with the smoothness of a function.
\subsection{Decay of fractional Fourier coefficients of integrable function}\label{section4.1}
\begin{proposition}\label{pro3.4}\,\
{\rm(Riemann-Lebesgue lemma)} Suppose $f\in e_{-\a}L^1(\mathbf{T}^n_\a)$.  We have
\begin{align*}
|\mathcal{F}_\a(f)(m)|\to 0 \quad as \quad|m|\to\infty.
\end{align*}
\begin{proof}
For all $\v>0$,  there exists a trigonometric polynomial $P_\a$ such that $\|f-P_\a\|_{L^1}<\v$. If $|m|>\text{degree}(P_\a)$, then  we have $\mathcal{F}_\a(P_\a)(m)= 0$. Hence
\begin{align*}
|\mathcal{F}_\a(f)(m)|=|\mathcal{F}_\a(f)(m)-\mathcal{F}_\a(P_\a)(m)|\leq \|f-P_\a\|_{L^1}<\v.
\end{align*}
This implies that $|\mathcal{F}_\a(f)(m)|\to 0$  as $|m| \to\infty$.
\end{proof}
\end{proposition}
For $f\in e_{-\a}L^1(\mathbf{T}^n_\a)$, we claim that $|\mathcal{F}_\a(f)(m)|$ may tend to zero arbitrarily slowly. More precise, we have the following theorem.
\begin{theorem}
Let $(d_m)_{m\in \mathbf{Z}^n}$ be a sequence of positive real numbers with $d_m\to 0$ as $|m|\to \infty$. Then there exists a function $ f\in e_{-\a}L^1(\mathbf{T}^n_\a)$ such that $|\mathcal{F}_\a(f)(m)|\geq d_m$ for all $m\in \mathbf{Z}^n$. In other words, given any rate of decay, there exists an integrable function on $\mathbf{T}^n_\a$ whose absolute value of fractional Fourier coefficients have slower rate of decay.
\end{theorem}
\begin{proof}
We first deal with $n=1$. Let $\{a_m\}_{m\in \mathbf{Z}}$ be a sequence of positive numbers that converges to zero as $|m|\to \infty$. Apply Lemma 3.3.2 in \cite{G} to the sequence $\{a_m+a_{-m}\}_{m\geq0}$ to find a convex sequence $\{c_m\}_{m\geq0}$ that dominates $\{a_m+a_{-m}\}_{m\geq0}$ and decreases to zero as $|m|\to \infty$. Extend $c_m$ for $m<0$ by taking $c_m = c_{|m|}$. Set
\begin{align}\label{eq3.3}
f(x)=e_{-\a}(x)\sum_{j=0}^\infty(j+1)(c_j+c_{j+2}-2c_{j+1})F_j^{1,\a}(x).
\end{align}
Using Lemma 3.3.3 in \cite{G} with $s =0$, we have
\begin{align}\label{eq3.4}
\sum_{j=0}^\infty(j+1)(c_j+c_{j+2}-2c_{j+1})\|F_j^{1,\a}\|_{L^1(\mathbf{T}^1_\a)}=c_0<\infty,
\end{align}
since $\|F_j^{1,\a}\|_{L^1(\mathbf{T}^1_\a)}=1$  for all $j$. Therefore, we obtain that $f\in{e_{-\a}L^1(\mathbf{T}^1_\a)}$ by noting that the series in (\ref{eq3.3}) converges in ${L^1(\mathbf{T}^1_\a)}$. For $m \in \mathbf{Z}$, we have
\begin{align}\label{eq3.5}
\nonumber|\mathcal{F}_\a(f)(m)|=&\sum_{j=0}^\infty(j+1)(c_j+c_{j+2}-2c_{j+1})\big|\mathcal{F}_\a (e_{-\a}F_j^{1,\a})(m)\big|\\
\nonumber=&\sum_{j=|m|}^\infty(j+1)(c_j+c_{j+2}-2c_{j+1})\Big|A_\a e_{\a}(m)\Big(1-\frac{|m|}{j+1}\Big)\Big|\\
\nonumber=&|\csc\a|^{1/2}\sum_{r=0}^\infty(r+1)(c_{r+|m|}+c_{r+|m|+2}-2c_{r+|m|+1})\\
=&|\csc\a|^{1/2}c_{|m|}=|\csc\a|^{1/2}c_m,
\end{align}
where the second equality follows from Proposition \ref{pro1.3} and the third equality is due to Lemma 3.3.3 in \cite{G} with $s = |m|$.

Next, we turn to $n\geq 2$. Let $(d_m)_{m\in \mathbf{Z}^n}$ be a sequence of positive real numbers with $d_m\to 0$ as $|m|\to \infty$. There exists a positive sequence $\{a_j\}_{j\in \mathbf{Z}}$ such that $a_{m_1}\cdots a_{m_n}\geq d_{({m_1},\ldots,{m_n})}$ and $a_j\to 0$ as $|j|\to \infty$. Set $$\mathbf{f}(x_1,\ldots,x_n)=f(x_1)\cdots f(x_n),$$
where $f$ is defined as in (\ref{eq3.3}) such that $|\mathcal{F}_\a (f)(m)|\geq a_m$. This together with (\ref{eq2.11}) implies $|\mathcal{F}_\a(\mathbf{f})(m)|\geq d_m$.
\end{proof}
\subsection{Decay of Fractional Fourier coefficients of smooth functions}\label{section4.2}
In this subsection, we are devoted to studying the relationship between the decay of fractional Fourier coefficients and the smoothness of a function.
\begin{definition}
For $0<\gamma<1$, the homogeneous Lipschitz space of order $\gamma$ on $\mathbf{T}^n_\a$ is defined by
 $$\dot{\Lambda}_\gamma(\mathbf{T}^n_\a)=\{f:\mathbf{T}^n_\a\to C\; with\; \|f\|_{\dot{\Lambda}_\gamma}<\infty\},$$
 where $$\|f\|_{\dot{\Lambda}_\gamma}:=\underset{h\neq0}{\sup_{x,h\in\mathbf{T}^n_\a}}\frac{|f(x+h)-f(x)|}{|h|^\gamma}.$$
\end{definition}
Next, we discuss the decay of fractional Fourier coefficients of Lipschitz functions.
\begin{theorem}\label{th3.7}
Let $s\in \mathbf{Z}$ and $s\geq 0$.\\
(a) Suppose that $f\in e_{-\a}C^s(\mathbf{T}^n_\a)$. Then
\begin{align}\label{eq3.7}
|\mathcal{F}_\a (f)(m)|\leq
\Big(\frac{\sqrt n}{2\pi |\csc\a|}\Big)^s\frac{\max\limits_{|\beta|=s}{|\mathcal{F}_\a [e_{-\alpha}\partial^\beta(e_{\alpha}f)](m)|}}{|m|^s}
\end{align}
and thus
\begin{align*}
|\mathcal{F}_\a(f)(m)|(1+|m|^s)\to 0
\end{align*}
as $|m|\to \infty$.\\
(b) Suppose that $f\in e_{-\a}C^s(\mathbf{T}^n_\a)$ and whenever $|\beta|=s$, $\partial^\beta(e_{\a}f)$ are in $\dot{\Lambda}_\gamma(\mathbf{T}^n_\a)$ for some $0<\gamma<1$. Then
\begin{align}\label{eq3.6}
|\mathcal{F}_\a (f)(m)|\leq\frac{(\sqrt n)^{s+\gamma}}{(2\pi)^s |\csc\a|^{s+\gamma+n-1}2^{\gamma+1}}\frac{\max\limits_{|\beta|=s}{\|\partial^\beta(e_{\a}f)
\|_{\dot{\Lambda}_\gamma}}}{|m|^{s+\gamma}}, \qquad m\neq0.
\end{align}
\end{theorem}
\begin{proof}
For fixed $m \in \mathbf{Z}^n \backslash\{0\}$, there exists a $j$ such that $|m_j| = \sup_{1\leq k\leq n} |m_k|$. Therefore, $|m|\leq \sqrt n |m_j|$. For $f\in e_{-\a}C^s(\mathbf{T}^n_\a)$, integrating by parts $s$ times with respect to the variable $x_j$, we get
\begin{align}\label{eq3.1}
\nonumber\mathcal{F}_\a(f)(m)
\nonumber=&A_\a^n e_{\a}(m)\int_{\mathbf{T}^n_\a}(e_{\a}f)(x)e^{-2\pi i (m\cdot x)\csc \a}dx\\
\nonumber=&A_\a^n e_{\a}(m)\int_{\mathbf{T}^n_\a}\partial_j^s(e_{\a}f)(x)\frac{e^{-2\pi i (m\cdot x)\csc \a}}{(2\pi i m_j\csc \a)^s}dx\\
\nonumber=&\frac{1}{(2\pi i m_j\csc \a)^s}\int_{\mathbf{T}^n_\a}e_{-\a}(x)\partial_j^s (e_{\a}f)(x)K_\a(m,x)dx\\
=&\frac{1}{(2\pi i m_j\csc \a)^s}\mathcal{F}_\a [e_{-\alpha}\partial_j^s(e_{\alpha}f)](m).
\end{align}
Then
$$|\mathcal{F}_\a(f)(m)|\leq
\Big(\frac{\sqrt n}{2\pi |\csc\a|}\Big)^s\frac{\max\limits_{|\beta|=s}{|\mathcal{F}_\a [e_{-\alpha}\partial^\beta(e_{\alpha}f)](m)|}}{|m|^s}.$$

We now turn to conclusion (b). Let {$e_j =(0,...,|\sin \a|,...,0)$} be the element of the torus $\mathbf{T}^n_\a$ whose $j$th coordinate is $|\sin \a|$ and all the others are zero. By noting that $e^{\pi i}=e^{-\pi i} =-1$, we have
\begin{align*}
\int_{\mathbf{T}^n_\a}\partial_j^s(e_{\a}f)(x)e^{-2\pi i (m\cdot x)\csc \a}dx
=-\int_{\mathbf{T}^n_\a}\partial_j^s(e_{\a}f)\Big(x-\frac{e_j}{2m_j}\Big)e^{-2\pi i (m\cdot x)\csc \a}dx.
\end{align*}
Therefore,
\begin{align*}
&\int_{\mathbf{T}^n_\a}\partial_j^s\big(e_{\a}f\big)(x)e^{-2\pi i (m\cdot x)\csc \a}dx\\
&\quad=\frac{1}{2}\int_{\mathbf{T}^n_\a}\Big[\partial_j^s\big(e_{\a}f\big)(x)
-\partial_j^s\big(e_{\a}f\big)\Big(x-\frac{e_j}{2m_j}\Big)\Big]e^{-2\pi i (m\cdot x)\csc \a}dx.
\end{align*}
Notice that
\begin{align*}
\Big|\partial_j^s(e_{\a}f)(t)-\partial_j^s(e_{\a}f)\Big(t-\frac{e_j}{2m_j}\Big)\Big|
\leq\frac{\|\partial_j^s(e_{\a}f)\|_{\dot{\Lambda}_\gamma}}{(2|m_j\csc\a|)^\gamma}.
\end{align*}
This combined with (\ref{eq3.1}) leads to
\begin{align*}
\big|\mathcal{F}_\a(f)(m)\big|
=&\Big|\frac{A_\a^n e_{\a}(m)}{(2\pi i m_j\csc \a)^s}\int_{\mathbf{T}^n_\a}\partial_j^s\big(e_{\a}f\big)(x)e^{-2\pi i m\cdot x\csc \a}dx\Big|\\
\leq&\Big(\frac{\sqrt n}{2\pi |\csc\a||m|}\Big)^s\frac{|A_\a^n| }{2|\csc\a|^n}\frac{\|\partial_j^s(e_{\a}f)\|_{\dot{\Lambda}_\gamma}}{(2|m_j\csc\a|)^\gamma}\\
\leq&\Big(\frac{\sqrt n}{2\pi |\csc\a||m|}\Big)^s\frac{1}{2|\csc\a|^{n/2}} \Big(\frac{\sqrt n}{2 |\csc\a||m|}\Big)^\gamma{\|\partial_j^s\big(e_{\a}f\big)\|_{\dot{\Lambda}_\gamma}}\\
\leq&\frac{(\sqrt n)^{s+\gamma}}{(2\pi)^s |\csc\a|^{s+\gamma+n/2}2^{\gamma+1}}\frac{\max\limits_{|\beta|=s}{\|\partial_j^s
(e_{\a}f)\|_{\dot{\Lambda}_\gamma}}}{|m|^{s+\gamma}},
\end{align*}
which completes the proof.
\end{proof}
The following corollary is a consequence of Theorem \ref{th3.7}.
\begin{corollary}
Let $s\in \mathbf{Z}$ with $s\geq 0$.\\
(a) Suppose that $f\in e_{-\a}C^s(\mathbf{T}^n_\a)$. Then
\begin{align*}
|\mathcal{F}_\a (f)(m)|\leq c_{n,s,\a}
\frac{\max(\|f\|_1,\max_{|\beta|=s}{\|\partial^\beta (e_\a f)\|_1})}{(1+|m|)^s}
\end{align*}
(b) Suppose that $f\in e_{-\a}C^s(\mathbf{T}^n_\a)$ and whenever $|\beta|=s$, $\partial^\beta (e_{\a}f) $ are in $\dot{\Lambda}_\gamma(\mathbf{T}^n_\a)$ for some $0<\gamma<1$. Then
\begin{align*}
|\mathcal{F}_\a (f)(m)|\leq& c_{n,s,\a}
\frac{\max(\|f\|_1,\max_{|\beta|=s}{\|\partial_j^s(e_{\a}f)\|_{\dot{\Lambda}_\gamma}})}{(1+|m|)^{s+\gamma}}.
\end{align*}
\end{corollary}
The following proposition provides a partial converse to Theorem \ref{th3.7}. We
denote below by $[[s]]$ the largest integer strictly less than a given real number s. Then $[[s]]$ is equal to the integer part $[s]$ of $s$, unless $s$ is an integer, in which case $[[s]] = [s]-1$.
\begin{proposition}
Suppose $s > 0$. If $f\in e_{-\a}L^1(\mathbf{T}^n_\a)$ and satisfies
{\begin{align}\label{eq3.11}
|\mathcal{F}_\a(f)(m)|\leq C(1+|m|)^{-s-n}
\end{align}}
for all $m \in\mathbf{ Z}^n$. Then $e_{\a}f$ has partial derivatives of all orders $|\beta|\leq [[s]]$. Moreover, $\partial^\beta (e_{\a}f)\in {\dot{\Lambda}_\gamma}$ for all multi-indices $\beta$ satisfying $|\beta| =[[s]]$ with $0 <\gamma<s-[[s]]$.
\end{proposition}
\bp
Using (\ref{eq3.11}) and Corollary \ref{cor3.3}, we have
\begin{align}\label{eq3.10}
f(x)=\sum_{ m\in \mathbf{Z}^n}\mathcal{F}_\a (f)(m) K_{-\a}(m,x)
\end{align}
for almost all $x\in \mathbf{T}^n_\a$. Let $f_m(x)=\mathcal{F}_\a (f)(m) K_{-\a}(m,x)$. If a series $g=\sum_m g_m$ satisfies
$\sum_m \|\partial^\beta g_m\|_\infty<\infty$ for all $|\beta|\leq M$, we obtain $g\in C^M$ and $\partial^\beta g=\sum_m \partial^\beta g_m$. Then, we need to check
$$\sum_{ m\in \mathbf{Z}^n}|\partial^\beta(e_{\a}f_m)(x)|<\infty.$$
Notice that
\begin{align*}
\partial^\beta(e_{\a}f_m)(x)
=&A_{-\a}^ne_{-\a}(m)\mathcal{F}_\a(f)(m)(2\pi i m\csc\a)^\beta e^{2\pi i (m\cdot x)\csc\a}.
\end{align*}
For $|\beta|\leq [[s]]$, by $(\ref{eq3.11})$, we get
\begin{align*}
\sum_{ m\in \mathbf{Z}^n}\big|\partial^\beta(e_{\a}f_m)(x)\big|
\leq & |A_{-\a}^n |\sum_{ m\in \mathbf{Z}^n}|\mathcal{F}_\a(f)(m)|
\sup_{x\in \mathbf{T}^n_\a}|(2\pi i m\csc\a)^\beta e_{-\a}(m,x)|
< \infty.
\end{align*}
Hence $e_{\a}f\in C^{[[s]]}(\mathbf{T}^n_\a)$, and
\begin{align*}
\partial^\beta (e_{\a}f)(x)
=& e_{\a}(x)\sum_{ m\in \mathbf{Z}^n}\mathcal{F}_\a(f)(m)(2\pi i m\csc\a)^\beta K_{-\a}(m,x).
\end{align*}
For $|\beta|=[[s]]$ and $0<\gamma<s-[[s]]$, we get
\begin{align*}
|e^{2\pi i (m\cdot h)\csc\a}-1|\leq \min(2,2\pi |m||h||\csc\a|)\leq 2^{1-\gamma}(2\pi)^\gamma|m|^\gamma|h|^\gamma|\csc\a|^\gamma.
\end{align*}
This together with (\ref{eq3.11}) implies
\begin{align*}
&\big|\partial^\beta (e_{\a}f)(x+h)-\partial^\beta (e_{\a}f)(x)\big|\\
=&\Big|A_{-\a}^n \sum_{ m\in \mathbf{Z}^n}\mathcal{F}_\a(f)(m)e_{- \a}(m)  e_{-\a}(m,x)\big(e^{2\pi i (m\cdot h)\csc\a}-1\big)(2\pi i m\csc\a)^\beta\Big|\\
\leq&\sum_{ m\in \mathbf{Z}^n} \frac{C|\csc\a|^{n/2}}{(1+|m|)^{s+n}} (2\pi)^{[[s]]} |m|^{[[s]]}|\csc\a|^{[[s]]}\big|e^{2\pi i (m\cdot h)\csc\a}-1\big|
\end{align*}
\begin{align*}
\leq&\sum_{ m\in \mathbf{Z}^n} \frac{C|\csc\a|^{n/2}}{(1+|m|)^{s+n}}(2\pi)^{[[s]]} |m|^{[[s]]}|\csc\a|^{[[s]]}2^{1-\gamma}(2\pi)^\gamma|m|^\gamma|h|^\gamma|\csc\a|^\gamma\\
\leq&2^{1-\gamma}|h|^\gamma(2\pi)^{[[s]]+\gamma}|\csc\a|^{[[s]]+n/2+\gamma}\sum_{ m\in \mathbf{Z}^n} \frac{C|m|^{[[s]]+\gamma}}{(1+|m|)^{s+n}}\\
\leq&2^{1-\gamma}|h|^\gamma(2\pi)^{s}|\csc\a|^{[[s]]+n/2+\gamma}\sum_{ m\in \mathbf{Z}^n} \frac{C|m|^{[[s]]+\gamma}}{(1+|m|)^{s+n}}\\
\leq&C2^{1-\gamma}(2\pi)^{s}|\csc\a|^{[[s]]+n/2+\gamma}|h|^\gamma.
\end{align*}
Therefore, $\partial^\beta (e_{\a}f)\in {\dot{\Lambda}_\gamma}$, and the proof is complete.
\ep
\section{Convergence of fractional Fourier series}
As we know, the convergence of Fourier series and the boundedness of singular integral and related operators have always been
the core subjects of harmonic analysis, see \cite{FHLL, FLPS, LFFY, WF, YFS1} and references therein.
In this section, we will discuss the pointwise convergence of fractional Fej\'{e}r means. For the convergence of other means, such as Bochner-Riesz means, one can refer to \cite{Bochner, FZ} and the references therein.
\subsection{Pointwise Convergence of the fractional Fej\'{e}r Means}\label{subsection5.1}
Using properties of the fractional Fej\'{e}r kernel, we obtain the following 1-dimensional result regarding the convergence of the fractional Fej\'{e}r means:
\begin{theorem}
If $f\in e_{-\a}L^1(\mathbf{T}^1_\a)$ has both left and right limits at a point $x_0$, denoted by $f(x_{0-})$ and $f(x_{0+})$, respectively, then
\begin{align}\label{eq3.16}
(f\overset{\alpha}{\ast}{{F}}^1_N)(x_0)\to \frac{1}{2}\big(f(x_{0+})+f(x_{0-})\big)\quad as \quad N\to \infty.
\end{align}
\end{theorem}
\bp
Set $\mathbf{T}^1_\a=[-{|\sin\a|}/{2},{|\sin\a|}/{2}]$. For fixed $\v>0$,  take $\d\in (0,{|\sin\a|}/{2})$ such that
\begin{align}\label{eq3.17}
\Big|\frac{(e_{\a}f)(x_{0}+t)+(e_{\a}f)(x_{0}-t)}{2}
-\frac{e_{\a}(x_0)f(x_{0+})+e_{\a}(x_0)f(x_{0-})}{2}\Big|<\frac{\v}{2},
\end{align}
where $0<t<\d$. Follows form Proposition \ref{pro2.4}, there exists $N>N_0$ such that
\begin{align}\label{eq3.18}
\sup_{\d\leq t\leq {|\sin\a|/2}}{{F}}_N^{1,\a}(t)<\frac{\v}{2},
\end{align}
Meanwhile,
\begin{align*}
&(f\overset{\alpha}{\ast}{{F}}_N^1)(x_0)-f(x_{0+})=e_{-\a}(x_0)\int_{\mathbf{T}^1_\a}
\big[(e_{\a}f)(x_0+t)-e_{\a}(x_0)f(x_{0+})\big]{{F}}_N^{1,\a}(t)dt,\\
&(f\overset{\alpha}{\ast}{{F}}_N^1)(x_0)-f(x_{0-})=e_{-\a}(x_0)\int_{\mathbf{T}^1_\a}
\big[(e_{\a}f)(x_0-t)-e_{\a}(x_0)f(x_{0-})\big]{{F}}_N^{1,\a}(t)dt.
\end{align*}
Using the fact that the integrand is even, we have
\begin{align}\label{eq5.4}
\nonumber(f\overset{\alpha}{\ast}{{F}}_N^1)(x_0)- \frac{f(x_{0+})+f(x_{0-})}{2}
&=2\int_{0}^{\frac{|\sin\a|}{2}}\Big[\frac{(e_{\a}f)(x_0+t)+(e_{\a}f)(x_0-t)}{2}
\\
&\quad\quad-\frac{e_{\a}(x_0)f(x_{0+})+e_{\a}(x_0)f(x_{0-})}{2}\Big]{{F}}_N^{1,\a}(t)
dt.
\end{align}
We divided the integral in (\ref{eq5.4}) into two parts, the integral over $[0,\delta)$ and the integral over $(\delta,{|\sin\a|}/{2})$, which are denoted by $I_3$ and $I_4$, respectively. Using (\ref{eq3.17}), we get $I_3\leq \v$. For $N\geq N_0$, by (\ref{eq3.18}), we obtain
\begin{align*}
I_4\leq\v\big(\|e_{\a}f-e_{\a}(x_0)f(x_{0_+})\|_{L^1}+\|e_{\a}f-e_{\a}(x_0)f(x_{0_-})\|_{L^1}\big)
&=:\v c(f,x_0,\a),
\end{align*}
where $c(f,x_0,\a)$ is a constant depending on $f$, $x_0$ and $\a$. This completes the proof.
\ep
\subsection{Almost everywhere convergence of the fractional Fej\'{e}r Means}\label{subsection5.2}

\begin{theorem}\label{th5.2}Suppose $f\in e_{-\a}L^1(\mathbf{T}^n_\a)$.\\
(a) Let
$$\mathcal{{H}}_{\a}f:=\sup_{N\in \mathbf{Z}^+}|f\overset{\alpha}{\ast}{{F}}^n_N|.$$
Then $\mathcal{{H}}_{\a}$ maps $e_{-\a}L^1(\mathbf{T}^n_\a)$ to $L^{1,\infty}(\mathbf{T}^n_\a)$ and $e_{-\a}L^p(\mathbf{T}^n_\a)$ to $L^p(\mathbf{T}^n_\a)$ for $1 < p \leq \infty$.\\
(b) For any function $f\in e_{-\a}L^1(\mathbf{T}^n_\a)$, we get that as $N\to \infty$
\begin{align*}
f\overset{\alpha}{\ast}{{F}}^n_N \to f \quad a.e.
\end{align*}
\end{theorem}
\bp
Using the fact that $1\leq \frac{|t|}{|\sin t|}\leq \frac{\pi}{2}$ when $|t|\leq \frac{\pi}{2}$, we obtain
\begin{align*}
|{{F}}_N^{1,\a}(x)|=&\frac{|\csc \a|}{N+1}\Big|\frac{\sin\big((N+1)\pi x \csc\a\big)} {\sin(\pi x \csc\a)}\Big|^2\\
\leq &\frac{(N+1)|\csc \a|}{4}\Big|\frac{\sin\big((N+1)\pi x \csc\a\big)} {(N+1) x \csc\a}\Big|^2\\
\leq &\frac{(N+1)|\csc \a|}{4}\min\Big\{\pi^2, \frac{1}{(N+1)^2|x\csc\a|^2}\Big\}\\
\leq &\frac{\pi^2}{2}\frac{(N+1)|\csc \a|}{1+(N+1)^2|x|^2}.
\end{align*}
where $|x|\leq \frac{1}{2|\csc\a|}$. For $t \in \mathbb{R}$, set $\varphi(t)=(1+t^2)^{-1}$ and $\varphi_\varepsilon (t)=\varepsilon^{-1}\varphi(t/\varepsilon)$ for $\varepsilon>0$. For $x=(x_1,\ldots,x_n)$ and $\v>0$, let
$$\Phi(x)=\varphi(x_1)\cdots\varphi(x_n)$$
and $\Phi_\varepsilon (t)={\varepsilon^{-n}}\Phi(t/\varepsilon)$. Therefore, for $|x|\leq {|\sin\a|/2}$ we get $|{{F}}_N^{1,\a}(x)|\leq {\pi^2}|\csc \a| \varphi_\varepsilon(x)$ with $\v=(N+1)^{-1}$. Meanwhile, for $y\in [-|\sin\a|/2,{|\sin\a|/2}]^n$, we have
$$|{{F}}_N^{n,\a}(y)|\leq \pi^{2n}|\csc \a|^n\Phi_\varepsilon (y),$$
where $\v=(N+1)^{-1}$.

Now let $e_{\a}f$ be an integrable function on ${\mathbf{T}^n_\a}$ and let $f_0$ denote its periodic extension on $\rn$. For $x\in [-|\sin\a|/2,{|\sin\a|/2}]^n$, we obtain
\begin{align*}
\mathcal{{H}}_\a f(x)
=&\sup_{N\in \mathbf{Z}^+}\Big|e_{-\a}(x)\int_{\mathbf{T}^n_\a}(e_{\a}f)(x-t){{F}}^{n,\a}_N(t)\mathrm{d}t\Big|\\
\leq&\pi^{2n}|\csc \a|^n\sup_{\v>0}\int_{\mathbf{T}^n_\a}|(e_{\a}f)(x-t)|\Phi_\varepsilon (t)\mathrm{d}t\\
\leq&10^n|\csc \a|^n\sup_{\v>0}\int_{\mathbf{T}^n_\a}{|f_0(x-t)|}\Phi_\varepsilon (t)\mathrm{d}t\\
\leq&10^n|\csc \a|^n\sup_{\v>0}\int_{\rn}{|(f_0\chi_Q)(x-t)|}\Phi_\varepsilon (t)\mathrm{d}t\\
=&:10^n|\csc \a|^n\mathfrak{F}(f_0\chi_Q)(x),
\end{align*}
where $Q$ is the cube $[-|\sin\a|,|\sin\a|]^n$ and $\mathfrak{F}$ is the operator on the set of integrable functions on $\rn$, which is defined by
$$\mathfrak{F}(h)=\sup_{\v>0}|h|\ast \Phi_\v.$$
By Lemma 3.4.5 in \cite{G}, we have
\begin{align*}
\|\mathcal{{H}}_{\a}f\|_{L^{1,\infty}(\mathbf{T}^n_\a)}
\leq10^n|\csc \a|^n\|\mathfrak{F}(f_0\chi_Q)\|_{L^{1,\infty}(\mathbb{R}^n)}
\leq C_{n,\a}\|f_0\chi_Q\|_{L^{1}(\rn)}
\leq C_{n,\a}\|f\|_{L^{1}(\mathbf{T}^n_\a)}.
\end{align*}
Meanwhile, it is easy to see
\begin{align*}
\|\mathcal{{H}}_{\a}f\|_{L^{\infty}(\mathbf{T}^n_\a)}\leq C_{n,\a}\|f\|_{L^{\infty}(\mathbf{T}^n_\a)}.
\end{align*}
By the Marcinkiewicz interpolation theorem, we obtain the $L^p$ conclusion.

Now we turn to conclusion (b). Since the sequence $\{{{F}}_N^{n,\a}\}_{N=0}^\infty$ is an approximate identity and $e_{-\a}C^\infty({\mathbf{T}^n_\a})$ is dense in $e_{-\a}L^1({\mathbf{T}^n_\a})$, we have
$$f\overset{\alpha}{\ast}{{F}}^n_N \to f, \quad f\in e_{-\a}C^\infty({\mathbf{T}^n_\a}),$$
uniformly on ${\mathbf{T}^n_\a}$ as $N\to \infty$. Since $\mathcal{{H}}_\a$ maps $e_{-\a}L^1({\mathbf{T}^n_\a})$ to $L^{1,\infty}({\mathbf{T}^n_\a})$ and using Theorem 2.1.14 in \cite{G}, we obtain that for $f\in e_{-\a}L^1({\mathbf{T}^n_\a})$,  $f\overset{\alpha}{\ast}{{F}}^n_N \to f$ a.e.
\ep
\subsection{Norm convergence of the fractional Bochner-Riesz Means}\label{subsection5.3}
\begin{theorem}\label{th5.3}
For $R>0$ and $m\in \mathbf{Z}^n$, let $a(m,R)$ be complex numbers such that\\
(1) For every $R>0$ there is a $q_R$ such that $a(m,R)=0$ if $|m|>q_R$;\\
(2) There is an $M_0<\infty$ such that $|a(m,R)|\leq M_0$ for all $m\in \mathbf{Z}^n$ and all $R>0$;\\
(3) For all $m\in \mathbf{Z}^n$, the limit of $a(m,R)$ exists as $R\to \infty$ and $lim_{R\to \infty}a(m,R)=a_m$.

Let $1\leq p<\infty$. For $f\in e_{-\a}L^p(\mathbf{T}^n_\a)$ and $x\in \mathbf{T}^n_\a$, define
\begin{align*}
S_R f(x)=\sum_{ m\in \mathbf{Z}^n}a(m,R)\mathcal{F}_\a (f)(m) K_{-\a}(m,x).
\end{align*}
For $h\in e_{-\a}C^\infty(\mathbf{T}^n_\a)$, define
\begin{align*}
A h(x)=\sum_{ m\in \mathbf{Z}^n}a_m \mathcal{F}_\a (f)(m) K_{-\a}(m,x).
\end{align*}
Then for all $f\in e_{-\a}L^p(\mathbf{T}^n_\a)$ the sequence $S_R f$ converges in $L^p(\mathbf{T}^n_\a)$ as $R\to \infty$ if and only if there exists a constant $K<\infty$ such that
\begin{align}\label{eq5.5}
\sup_{R>0}\|S_R\|_{L^p\to L^p}\leq K.
\end{align}
Furthermore, if (\ref{eq5.5}) holds, then for the same constant $K$ and $h\in e_{-\a}C^\infty(\mathbf{T}^n_\a)$, we have
 \begin{align}\label{eq5.6}
\sup_{h\neq 0}\|Ah\|_{L^p\to L^p}\leq K,
\end{align}
and the $A$ extends to a bounded operator $\widetilde{A}$ from $e_{-\a}L^p(\mathbf{T}^n_\a)$ to $L^p(\mathbf{T}^n_\a)$; moreover, for every $ f\in e_{-\a}L^p(\mathbf{T}^n_\a)$ we have $S_R f\to \widetilde{A}f$ in $L^p$ as $R\to \infty$.
\end{theorem}

\begin{proof}
If $S_R f$ converges to in $L^p(\mathbf{T}^n_\a)$, then $\|S_R f\|_{L^p}\leq C_f$, where constant $C_f$ depends on $f$.  Moreover, each $S_R f$ is a bounded operator  from
$e_{-\a}L^p(\mathbf{T}^n_\a)$ to $L^p(\mathbf{T}^n_\a)$ with norm at most $\#\{m\in \mathbf{Z}^n : |m|\leq q_R\}M_0$. Therefore, we obtain that $\{S_R\}_{R>0}$ is a family of $L^p$ bounded linear operators that satisfy $\sup_{R>0}\|S_R\|_{L^p(e_{-a}\mathbf{T}^n_\a)\to L^p(\mathbf{T}^n_\a)}\leq C_f$ for each $f\in e_{-\a}L^p(\mathbf{T}^n_\a)$. It follows form the uniform boundedness theorem that the operator norms of $S_R$ are bounded uniformly in $R$. This proves $(\ref{eq5.5})$.

Conversely, assume $(\ref{eq5.5})$. For $h\in e_{-\a}C^\infty(\mathbf{T}^n_\a)$, by property (ii) and the Lebesgue dominated convergence theorem, we get
\begin{align*}
\lim_{R\to \infty}\sum_{ m\in \mathbf{Z}^n}a(m,R) \mathcal{F}_\a (f)(m) K_{-\a}(m,x)=\sum_{ m\in \mathbf{Z}^n}a_m \mathcal{F}_\a (f)(m) K_{-\a}(m,x).
\end{align*}
Fatou's lemma gives
\begin{align*}
\|Ah\|_{L^p}=\|\lim_{R\to \infty} S_R h\|_{L^p}\leq K\| h\|_{L^p},
\end{align*}
hence (\ref{eq5.6}) holds. Thus $A$ extends to a bounded operator $\widetilde{A}$ on $e_{-\a}L^p(\mathbf{T}^n_\a)$ by density.

For $f\in e_{-\a}L^p(\mathbf{T}^n_\a)$, we need to prove $S_R f\to \widetilde{A}f$ in $L^p$ as $R\to \infty$. For all $\v>0$,  there exists a trigonometric polynomial $P_\a$ such that $\|f-P_\a\|_{L^p}\leq \v$. Notice that $\mathcal{F}_\a(P_\a)(m)= 0$  for $|m|>\text{degree}(P_\a)$. Then there is an $R_0>0$ such that for all $R > R_0$ we have
\begin{align*}
\sum_{|m_1|+\cdots+|m_n|\leq \text{degree}(P_\a)}|a(m,R)-a_m|\mathcal{F}_\a(P_\a)(m)\leq\v.
\end{align*}
For  $R > R_0$, we deduce that
\begin{align*}
\|S_R P_\a- A P_\a\|_{L^p}
&\leq \|S_R P_\a- A P_\a\|_{L^\infty}\\
&\les \sum_{|m_1|+\cdots+|m_n|\leq \text{degree}(P_\a)}|a(m,R)-a_m|\mathcal{F}_\a(P_\a)(m)
\les\v.
\end{align*}
Then
\begin{align*}
\|S_R f- \widetilde{A} f\|_{L^p}
&\leq \|S_R f- S_R P_\a\|_{L^p}+\|S_R P_\a- \widetilde{A} P_\a\|_{L^p}+\|\widetilde{A} P_\a- \widetilde{A} f\|_{L^p}\\
&\les (2K+1)\v
\end{align*}
for $R > R_0$. This completes the proof.
\end{proof}
\begin{definition}
The Bochner-Riesz means of order $\gamma\geq 0$ is the operator
\begin{align*}
B^\gamma_R(f)(x)=\sum_{ m\in \mathbf{Z}^n, |m|\leq R}\Big(1-\frac{|m|^2}{R^2}\Big)^\gamma \mathcal{F}_\a (f)(m) K_{-\a}(m,x).
\end{align*}
\end{definition}
In fact, suppose $a(m,R)=\Big(1-\frac{|m|^2}{R^2}\Big)^\gamma$ for $|m|\leq R$. It is easy to check that  the sequence $\Big(1-\frac{|m|^2}{R^2}\Big)^\gamma$ satisfy the properties in Theorem \ref{th5.3}. Hence, we have following corollary.
\begin{corollary}
Let $1\leq p<\infty$ and $\gamma\geq 0$. For $f\in e_{-\a}L^p(\mathbf{T}^n_\a)$, we have
\begin{align*}
\lim_{R\to\infty}\|B^\gamma_R(f)-f\|_{L^p}=0 \Longleftrightarrow \sup_{R> 0}\|B^\gamma_R\|_{L^p\to L^p}.
\end{align*}
\end{corollary}
\section{Applications to Partial Differential Equations}
It is well known that Fourier theory play an important role in the solution of differential equations, see \cite{DFX, SZH, SZW, YFL, YFS} and the references therein for more details. In this section, we apply the fractional Fourier series to get solutions of the fractional heat equation and fractional Dirichlet problem.
\begin{definition}
For fixed $k>0$, define the fractional heat kernel of order $\a$:
$$H_{t}^\a(x)=\sum_{m\in \mathbf{Z}^n}e^{-4\pi^2|m|^2|\csc\a|^2kt}e^{2\pi im\cdot x},$$
where $t>0$. Note that $H_{t}^\a$ is absolutely convergent for any $t>0$.
\end{definition}

\begin{proposition}
Let $k>0$ be fixed and $f\in e_{-\a}C^\infty(\mathbf{T}^n_\a)$. Then the fractional heat equation of order $\a$
\begin{align}\label{eq6.1}
\displaystyle\frac{\partial}{\partial t}\Big(e_\a(x)F_\a(x,t)\Big)=k\triangle_x \Big(e_\a(x)F_\a(x,t)\Big),\quad t\in (0,\infty), x\in\mathbf{T}^n_\a,
\end{align}
under the initial condition
\begin{align}\label{eq6.111}
F_\a(x,0)=f(x),\quad x\in\mathbf{T}^n_\a,
\end{align}
has a unique solution which is continuous on $[0,\infty)\times \mathbf{T}^n_\a$ and $C^\infty$ on $[0,\infty)\times \mathbf{T}^n_\a$ given by
\begin{align}\label{eq6.2}
 F_\a(x,t)=(f\overset{\alpha}{\ast}{H_{t}^\a})(x)
=\sum_{m\in \mathbf{Z}^n}\mathcal{F}_\a (f)(m)e^{-4\pi^2|m|^2|\csc\a|^2kt}K_{-\a}(m,x).
\end{align}
\end{proposition}
\begin{proof}
For $f\in e_{-\a}C^\infty(\mathbf{T}^n_\a)$, it is easy to see that the series in (\ref{eq6.2}) is rapidly convergent in $m$ and we obtain that $F_\a(x,t)$ is a continuous function on $[0,\infty)\times \mathbf{T}^n_\a$. In addition, the series can be differentiated term by term with $t > 0$, and thus it produces a $C^\infty$ function on $[0,\infty)\times \mathbf{T}^n_\a$. By Corollary \ref{cor3.3}, $F_\a$ satisfies the initial condition. Next, it remains to check (\ref{eq6.1}).
\begin{align*}
\frac{\partial}{\partial t}\Big(e_\a(x)F_\a(x,t)\Big)=&\frac{\partial}{\partial t}\Big[\sum_{m\in \mathbf{Z}^n}\mathcal{F}_\a (f)(m)e^{-4\pi^2|m|^2|\csc\a|^2kt}A_{-\a}^n e_{-\a}(m)e_{-\a}(m,x)\Big]\\
=&k\Big[\sum_{m\in \mathbf{Z}^n}\mathcal{F}_\a (f)(m)e^{-4\pi^2m^2|\csc\a|^2kt}A_{-\a}^n e_{-\a}(m)\sum_{j=1}^n\frac{\partial^2}{\partial {x_j^2}}e_{-\a}(m,x)\Big]\\
=&k\sum_{j=1}^n\frac{\partial^2}{\partial {x_j^2}}\Big[\sum_{m\in \mathbf{Z}^n}\mathcal{F}_\a (f)(m)e^{-4\pi^2m^2|\csc\a|^2kt}A_{-\a}^n e_{-\a}(m)e_{-\a}(m,x)\Big]\\
=&k\triangle_x \Big(e_\a(x)F_\a(x,t)\Big),
\end{align*}
where the last equality follows from the convergence of the series.

Finally, assume that there is another solution $G_\a(x,t)$, which is continuous on $[0,\infty)\times {\mathbf{T}^n_\a}$ and $C^\infty$ on $[0,\infty)\times {\mathbf{T}^n_\a}$. Then $G_\a(x,t)$ can be expanded in fractional Fourier series as follows:
$$G_\a(x,t):=\sum_{m\in \mathbf{Z}^n}c_{m}^\a(t)K_{-\a}(m,x),$$
where $$c_{m}^\a(t)=\int_{\mathbf{T}^n_\a}G_\a(y,t)K_\a(m,y)dy.$$
Then $c_{m}^\a(t)$ is a smooth function on $(0,\infty)$ since $G_\a$ is $C^\infty$ on $[0,\infty)\times {\mathbf{T}^n_\a}$.
By equation (\ref{eq6.1}), we have
\begin{align*}
\frac{d}{dt}{c}_{m}^\a(t)=&\int_{\mathbf{T}^n_\a}\frac{\partial}{\partial t}\Big(e_\a(x)G_\a(x,t)\Big)e_{-\a}(x)K_\a(m,x)dx\\
=&k A_\a^n e_{\a}(m) \sum_{j=1}^n\int_{\mathbf{T}^n_\a}\frac{\partial^2}{\partial {x_j^2}}\Big(e_\a(x)G_\a(x,t)\Big)e^{-2\pi i(m\cdot x)\csc\a}dx\\
=&-4\pi^2|m|^2|\csc\a|^2k{c}_{m}^\a(t),
\end{align*}
where the last equality follows from an integration by parts in which the boundary terms cancel each other in view of the periodicity of the integrand in $x$.
In addition, ${c}_{m}^\a(0)=\mathcal{F}_\a(f)(m)$. Therefore, we get
$${c}_{m}^\a(t)=\mathcal{F}_\a(f)(m)e^{-4\pi^2|m|^2|\csc\a|^2kt}.$$
Thus $G_\a=F_\a$ on $[0,\infty)\times \mathbf{T}^n_\a$.
\end{proof}
\begin{remark}
Let $\mathbb{H}_{t}^\a(x)=|\csc\a|^n H_{t}^\a(x\csc\a)$.
The family $\big\{\mathbb{H}_{t}^\a\big\}_{t>0}$ is an approximate identity on $\mathbf{T}^n_\a$.
\end{remark}
\begin{definition}
Define the fractional Poisson kernel of order $\a$
$$P_{t}^\a(x)=\sum_{m\in \mathbf{Z}^n}e^{-2\pi|m|{|\csc\a|}t}e^{2\pi im\cdot x},$$
where $t>0$. Note that $P_{t}^\a$ is absolutely convergent for any $t>0$.
\end{definition}

\begin{proposition}
Let $f\in e_{-\a}C^\infty(\mathbf{T}^n_\a)$. Then the fractional Dirichlet problem of order $\a$
\begin{align}\label{eq6.3}
\begin{cases}
\displaystyle\frac{\partial^2}{\partial t^2}\Big(e_\a(x)F_\a(x,t)\Big)+\sum_{j=1}^n\frac{\partial^2}{\partial {x_j}^2}\Big(e_\a(x)F_\a(x,t)\Big)=0;\; t\in (0,\infty), x\in\mathbf{T}^n_\a,\\
\quad\quad\quad\quad\quad\quad\quad\quad\;\; F_\a(x,0)=f(x);\quad x\in\mathbf{T}^n_\a,
\end{cases}
\end{align}
has a solution which is continuous on $[0,\infty)\times \mathbf{T}^n_\a$ and $C^\infty$ on $[0,\infty)\times \mathbf{T}^n_\a$ given by
\begin{align*}
F_\a(x,t)=(f\overset{\alpha}{\ast}{P_{t}^\a})(x)=\sum_{m\in \mathbf{Z}^n}\mathcal{F}_\a (f)(m)e^{-2\pi|m|{|\csc\a|}t}K_{-\a}(m,x).
\end{align*}
\end{proposition}
\begin{proof}
It is easy to see that $F_\a(x,t)$ is a continuous function on $[0,\infty)\times \mathbf{T}^n_\a$ and $C^\infty$ on $[0,\infty)\times \mathbf{T}^n_\a$. By Corollary \ref{cor3.3}, $F_\a$ satisfies the initial condition. It suffices to verify (\ref{eq6.3}). Notice that
\begin{align*}
\frac{\partial^2}{\partial t^2}\Big(e_\a(x)F_\a(x,t)\Big)
=&-\sum_{m\in \mathbf{Z}^n}\mathcal{F}_\a (f)(m)e^{-2\pi|m|{|\csc\a|}t}A_{-\a}^n e_{-\a}(m^2)\sum_{j=1}^n\frac{\partial^2}{\partial {x_j^2}}e_\a(m,x)\\
=&-\sum_{j=1}^n\frac{\partial^2}{\partial {x_j^2}}\Big[ e_{\a}(x^2)\sum_{m\in \mathbf{Z}^n}\mathcal{F}_\a (f)(m)e^{-2\pi|m|{|\csc\a|}t}K_{-\a}(m,x)\Big]\\
=&-\triangle_x \Big(e_\a(x)F_\a(x,t)\Big).
\end{align*}
This completes the proof.
\end{proof}
\begin{remark}
Let $\mathbb{P}_{t}^\a(x)=|\csc\a|^n P_{t}^\a(x\csc\a)$.
The family $\{\mathbb{P}_{t}^\a\}_{t>0}$ is an approximate identity on $\mathbf{T}^n_\a$.
\end{remark}

\section{Applications to non-stationary signals}
In this section, we will demonstrate the use of Corollary \ref{cor3.3} and Theorem \ref{th5.2} in the recovery of non-stationary signals. For example, suppose
\[g(x)=\displaystyle\left\{\begin{array}{ll}
x,&-\frac{1}{4}<x\leq \frac{1}{4} ,\\
\frac{1}{4},\,&x=-\frac{1}{4},
\end{array}\right.\]
and $g(x+1/4)=g(x)$ for all $x\in \mathbb{R}$. By taking $\a=\pi/6$, we have
\begin{equation*}
f(x)=e^{-\pi ix^2 \cot\pi/6}g(x)=
\begin{cases}
xe^{-\sqrt3 \pi ix^2},&-\frac{1}{4}<x\leq \frac{1}{4},\\
\frac{1}{4}e^{-\frac{\sqrt3 }{16}\pi i},&x=-\frac{1}{4}.
\end{cases}
\end{equation*}
It is easy to see $f\in e_{-\pi/6}L^1(\mathbf{T}^1_{-\pi/6})$. Note that
\[
\mathcal{F}_\a (f)(m)=\left\{
\begin{array}
[c]{cc}%
\displaystyle iA_{\a}^1e^{\sqrt3\pi i m^2}\frac{(-1)^{m}}{8\pi m}, & m\neq 0,\\
\displaystyle 0, & m=0.
\end{array}
\right.
\]
Therefore, the series $\sum_{m\in\mathbf{Z}}|\mathcal{F}_\a (f)(m)|$ is not convergent and we can not recover the signal $f(x)$ by fractional Fourier inversion. Meanwhile, the approximating method (Theorem \ref{th5.2}) can give another way to recover the non-stationary signal $f(x)$. In fact, for $N\geq 0$, we obtain
\begin{align*}\label{eq1}
(f\overset{\pi/6}{\ast}{F}_N^1)(x)
=&ie^{-\sqrt3 \pi ix^2}\sum_{m=-N}^{N}\Big(1-\frac{|m|}{N+1}\Big)\frac{(-1)^{m}}{4\pi m}e^{4\pi i m x}.
\end{align*}
Figure \ref{fig:ue} shows the real and imaginary part graph of the fractional convolution $f\overset{\pi/6}{\ast}{F}_N^1$ with $N=10,50,100,500$.
\begin{figure}[htb]
\centering
\subfigure[real part graph of $(f\overset{\pi/6}{\ast}{F}_N^1)(x)$]{
\includegraphics[width=0.65\linewidth]{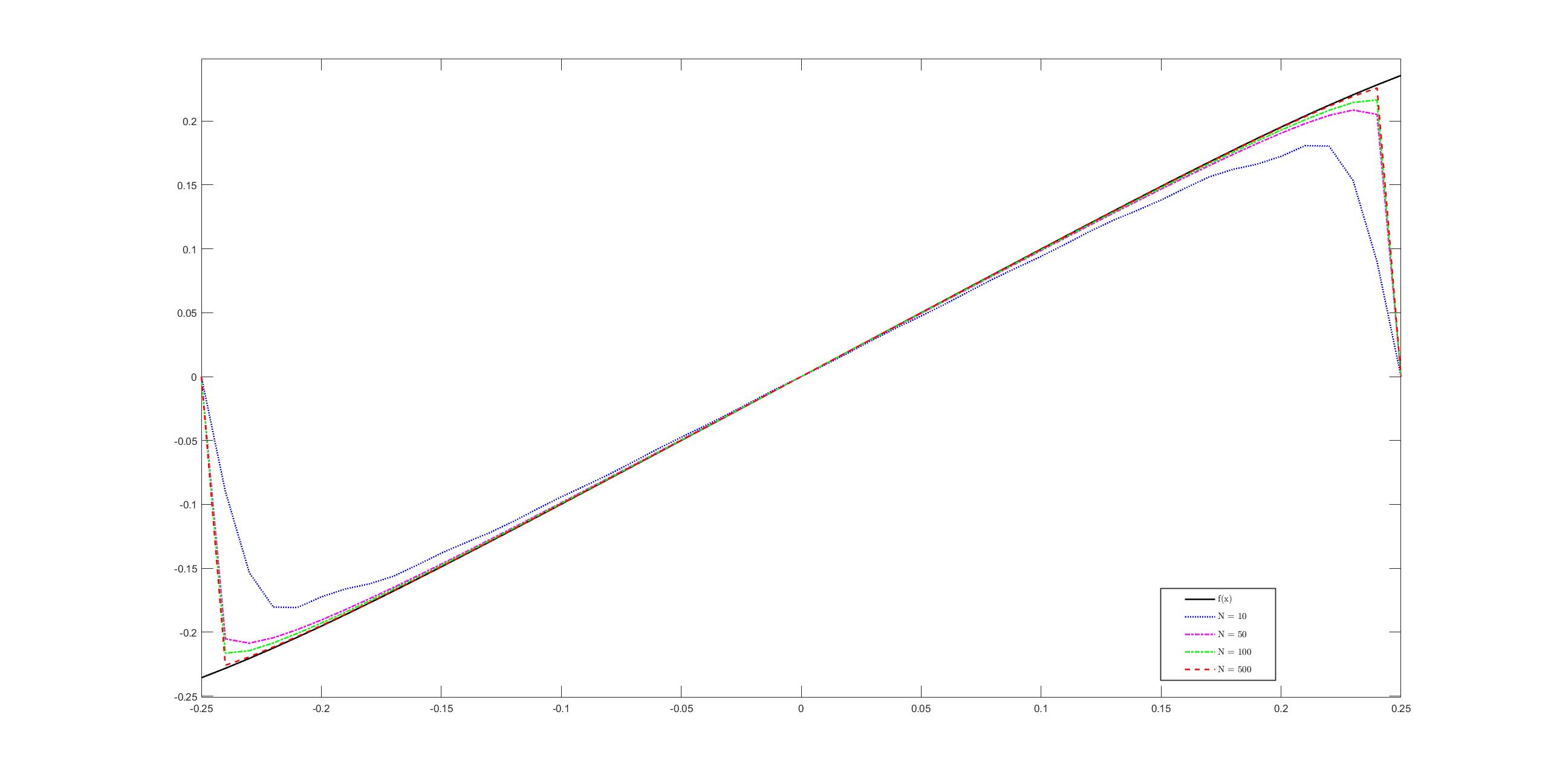}} \hspace{0.5em}
\subfigure[imaginary part graph of $(f\overset{\pi/6}{\ast}{F}_N^1)(x)$]{
\includegraphics[width=0.65\linewidth]{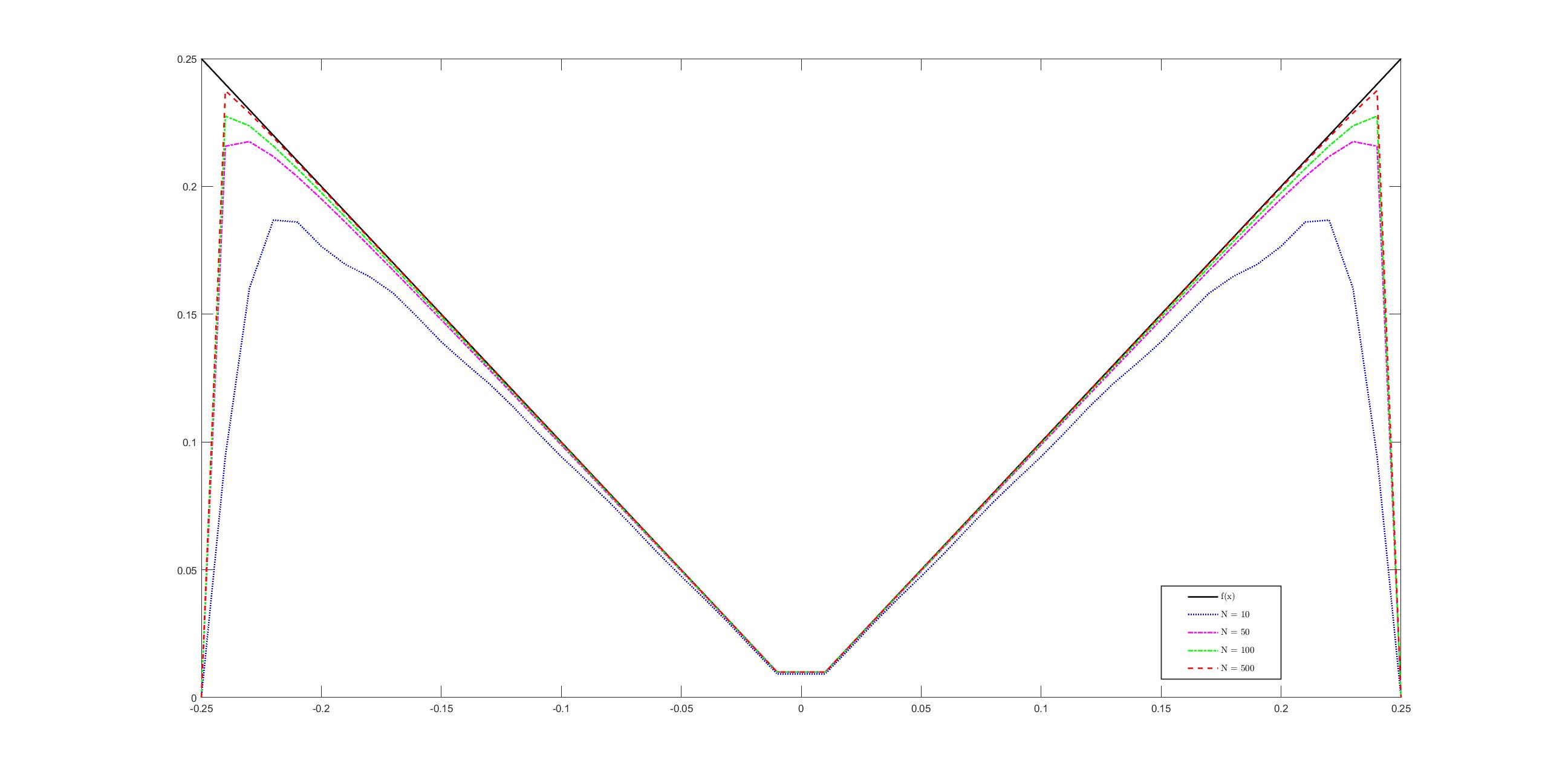}} \caption{real and imaginary
part graphs of $(f\overset{\pi/6}{\ast}{F}_N^1)(x)$}%
\label{fig:ue}%
\end{figure}

By Theorem \ref{th5.2}, we obtain that $(f\overset{\pi/6}{\ast}{F}_N^1)(x)\to f$ for almost every $x\in \mathbf{T}^1_{-\pi/6}$ as $N\to \infty$.

\bigskip
{\bf Acknowledgement:} This work was supported by the National Natural Foundation of China (Nos. 12301118, 12071197, 12171221, 12271232) and the Natural Science Foundation of Shandong Province (Nos. ZR2021MA031, ZR2021MA079).

\bibliographystyle{amsplain}

\end{document}